\documentclass[11pt]{article}
\newtheorem{theorem}{Theorem}
\newtheorem{proposition}{Proposition}
\newtheorem{lemma}{Lemma}
\newtheorem{proof}{Proof}
\usepackage{hyperref}

\usepackage{graphicx}

\usepackage{amsmath}
\usepackage{amssymb}
\usepackage{graphicx}
\usepackage[dvipsnames,usenames]{color}
\usepackage{multirow}
\usepackage{array}
\title{Analysis of the Symmetric Join the Shortest Orbit Queue}

\author{Ioannis Dimitriou\\ 
Department of Mathematics, 
University of Patras, P.O. Box 
26500, Patras, Greece\\
E-mail: idimit@math.upatras.gr\\Website: \href{https://thalis.math.upatras.gr/~idimit/}{https://thalis.math.upatras.gr/~idimit/}}
\begin{document}
\maketitle

\begin{abstract}
This work introduces the join the shortest queue policy in the retrial setting. We consider a Markovian single server retrial system with two infinite capacity orbits. An arriving job finding the server busy, it is forwarded to the least loaded orbit. Otherwise, it is forwarded to an orbit randomly. Orbiting jobs of either type retry to access the server independently. We investigate the stability condition, the stationary tail decay rate, and obtain the equilibrium distribution by using the compensation method.\vspace{2mm}
\end{abstract}

%\end{frontmatter}

%\linenumbers

\section{Introduction}
In this work, we introduce the concept of the \textit{join the shortest queue} (JSQ) policy in the retrial setting. We consider a single server retrial system with two infinite capacity orbit queues. The service station can handle at most one job. Thus, an arriving job that find the server busy joins the least loaded orbit queue, and in case of a tie, the job joins either orbit queue with probability $1/2$. Orbiting jobs retry independently to connect with the service station after a random time period. Our model can be seen as a special type of a queueing network that contains three nodes: the main service node where blocking is possible, and two delay nodes for repeated trials. External arrivals are routed initially to the main service node. If the main service node is empty, the arriving job start service immediately, and leaves the network after the service completion. Otherwise, the arriving blocked job is routed to the least loaded delay node (i.e., the least loaded orbit queue), and becomes a source of retrial job. Our primary aim is to investigate its stationary behaviour by using the compensation method.

The (non-modulated) two-dimensional JSQ problem was initially studied in \cite{haight,king}, A compact mathematical method using generating functions was provided in \cite{coh1,fay1}. However, it does not lead to an explicit characterization of the equilibrium probabilities. In \cite{ad0,ad2,ad3}, the authors introduced the compensation method (CM), an elegant and direct method to obtain explicitly the equilibrium join queue-length distribution as infinite series of product form terms; for other related important works, see also \cite{Adan2016,boxhout,sax,stella}. In \cite{ad1}, the CM was used to analyze the JSQ policy in a two-queue system fed by Erlang arrivals. The queueing system in \cite{ad1} is described by a multilayer random walk in the quarter plane; see also \cite{selen}.\vspace{-0.1in}%Numerical/approximation methods were also applied: see the power series algorithm, e.g., \cite{blanc1987,blanc1992}, and the matrix geometric method; see e.g., \cite{gerts,raoposner}, for which connections with CM was recently reported in \cite{stella}.\vspace{-0.1in} %; see also \cite{Dorsman2013}.,Hout}  for more complicated models. %However, PSA is numerically satisfactory for relatively lower dimensional models, although, the theoretical foundation of this method is still incomplete. %By expressing the equilibrium distribution as a power series in some variable based on the model parameters (usually the load), PSA transforms the balance equations into a recursively solvable set of equations by adding one dimension to the state space.  
%For the (non-modulated) multidimensional JSQ model we mention, the authors in \cite{Houtum} constructed upper and lower bounds for any performance measure based on two related systems that are easier to analyse. For a comparative analysis of the methods used for the analysis of multidimensional queueing models (including JSQ) see \cite{Onno}.
%\subsection{Our contribution}
\paragraph{Fundamental contribution:} In this work, we provide an exact analysis that unifies two queueing models: The JSQ model and the two-class retrial model. Our primary aim is to extend the applicability of the CM to random walks in the quarter plane modulated by a two-state Markov process, and in particular to retrial systems with two infinite capacity orbit queues; see Section \ref{sec:model}. In such a case, the phase process represents the state of the server, and affects the evolution of the level process, i.e, the orbit queue lengths in two ways: i) The rates at which certain transitions in the level process occur depend on the state of the phase process. Thus, a change in the phase might not immediately trigger a transition of the level process,
but changes its dynamics (indirect interaction). ii) A phase change does trigger an immediate transition of the level process (direct interaction). For this modulated two-dimensional random walk we investigate its stationary behaviour by using the \textit{compensation method}. We further study the \textit{stability condition} and investigate its \textit{stationary tail decay rate}. Our work, joint with \cite{ad1,selen}, indicate that CM can be extended to a class of multi-layered two-dimensional random walks that satisfy similar criteria as in \cite{ad3}. Thus, it is worth investigating under which conditions such an elegant solution is found through CM. This work serves as a first step towards extending the application of CM to JSQ retrial model with more than two orbits.\vspace{-0.1in}
\paragraph{Application oriented contribution:} Applications of this model can be found in relay-assisted cooperative communication system: There is a source user that transmits packets to a common destination node (i.e., the single service station), and a finite number of relay nodes (i.e., the orbit queues) that assist the source user by retransmitting its blocked packets, e.g., \cite{dim3,PappasTWC2015}. The JSQ protocol serves as the cooperation strategy among the source and the relays, under which, the user chooses to forward its blocked packet to the least loaded relay node. This works serves as a first step towards the analysis of even general retrial models operating under the JSQ policy.

The paper is organized as follows. In Section \ref{sec:model} we describe the model in detail, and investigate the necessary stability condition. Some useful preliminary results are presented in Section \ref{sec:decay}. The main result of this work that refers to the three-dimensional CM applied in a retrial network is given in Section \ref{sec:compe}. A numerical example is presented in Section \ref{sec:num}.\vspace{-0.2in}
 \section{Model description and stability condition}\label{sec:model}\vspace{-0.1in}
Consider a single server retrial system with two infinite capacity orbit queues. Jobs arrive at the system according to a Poisson process with rate $\lambda>0$. If an arriving job finds the server idle, it starts service immediately. Otherwise, it is routed to the least loaded orbit queue. In case both orbit queues have the same occupancy, the blocked job is routed to either orbit with probability 1/2. Orbiting jobs of either type retry independently to occupy the server after an exponentially distributed time period with rate $\alpha$, i.e., we consider the \textit{constant retrial policy}. Service times are independent and exponentially distributed with rate $\mu$. Let $Y(t)=\{(Q_{1}(t),Q_{2}(t),C(t)),t\geq0\}$, where $Q_{l}(t)$ the number of jobs stored in orbit $l$, $l=1,2,$ at time $t$, and by $C(t)$ the state of the server, i.e., $C(t)=1$, when it is busy and $C(t)=0$ when it is idle at time $t$, respectively. $Y(t)$ is an irreducible Markov process on $\{0,1,\ldots\}\times\{0,1,\ldots\}\times\{0,1\}$. Denote by $Y=\{(Q_{1},Q_{2},C)\}$ its stationary version.  Define the set of stationary probabilities $p_{i,j}(k)=\mathbb{P}(Q_{1}=i,Q_{2}=j,C=k)$. Let $J$ be a random variable indicating the orbit queue which an arriving blocked job joins. Clearly, $J$ is dependent on the vector $(Q_{1},Q_{2},C)$. Then, the equilibrium equations are
\begin{equation}
\begin{array}{r}
p_{i,j}(0)(\lambda+\alpha (1_{\{i>0\}}+1_{\{j>0\}}))=\mu p_{i,j}(1),\\
p_{i,j}(1)(\lambda+\mu)=\lambda p_{i,j}(0)+\alpha[p_{i+1,j}(0)+p_{i,j+1}(0)]\\+\lambda[p_{i-1,j}(1)\mathbb{P}(J=1|Q=(i-1,j,1))1_{\{i>0\}}+p_{i,j-1}(1)\mathbb{P}(J=2|Q=(i,j-1,1))1_{\{j>0\}}],
\end{array}\label{z0}
\end{equation}
where, $1_{\{E\}}$ the indicator function of the event $E$, $\mathbb{P}(J=l|Q=(Q_{1},Q_{2},1))=1_{\{Q_{l}\leq Q_{k}\}}/[\sum_{k=1}^{2}1_{\{Q_{k}=Q_{l}\}}]$
%\begin{displaymath}
%\begin{array}{c}
%P(J=1|Q=(i,j,1))=\left\{\begin{array}{ll}
%1,&i<j,\\
%1/2,&i=j,\\
%0,&i>j,
%\end{array}\right.\,\,P(J=2|Q=(i,j,1))=\left\{\begin{array}{ll}
%1,&i>j,\\
%1/2,&i=j,\\
%0,&i<j,
%\end{array}\right.
%\end{array}
%\end{displaymath}
and $\mathbb{P}(J=l|Q=(Q_{1},Q_{2},0))=0$, $l=1,2$. We first establish necessary conditions, which also turn out to be sufficient.% The following proposition states that $\lambda/\mu\leq 1$ and $(\lambda/\mu)(\lambda+2\alpha)/(2\alpha)\leq 1$
\begin{proposition}\label{prop10}
\begin{equation*}
\begin{array}{c}
\mathbb{P}(C=1)=\frac{\lambda}{\mu},\,p_{0,.}(0)=\sum_{j=0}^{\infty}p_{0,j}(0)=p_{.,0}(0)=\sum_{i=0}^{\infty}p_{i,0}(0)=1-\frac{\lambda(\lambda+2\alpha)}{2\alpha\mu}.
\end{array}
\end{equation*}
\end{proposition}
\begin{proof}
Note that due to the symmetry of the model $p_{.,0}(0)=p_{0,.}(0)$. For each $i= 0, 1, 2,\ldots$ we consider the cut between the the states $\{Q_{1} = i, C = 1\}$ and $\{Q_{1} = i+1, C = 0\}$. This yields $\lambda \mathbb{P}(J=1,Q_{1}=i,C=1)=\alpha p_{i+1,.}(0).$ Summing for all $i=0,1,\ldots$ results in $\lambda \mathbb{P}(J=1,C=1)=\alpha[\mathbb{P}(C=0)-p_{0,.}(0)]$. Similarly, $\lambda \mathbb{P}(J=2,C=1)=\alpha[\mathbb{P}(C=0)-p_{.,0}(0)]$. Summing the last two expressions yields
\begin{equation}
\begin{array}{c}
\lambda \mathbb{P}(C=1)=2\alpha \mathbb{P}(C=0)-\alpha[p_{0,.}(0)+p_{0,.}(0)]=2\alpha \mathbb{P}(C=0)-2\alpha p_{0,.}(0),
\end{array}\label{z4}
\end{equation}
where the second equality is due to the symmetry of the model. Summing the first in \eqref{z0} for all $i,j=0,1,\ldots$ yields
\begin{equation}
\begin{array}{c}
(\lambda+2\alpha)\mathbb{P}(C=0)-\mu \mathbb{P}(C=1)=\alpha[p_{.,0}(0)+p_{0,.}(0)]=2\alpha p_{0,.}(0).
\end{array}\label{z5}
\end{equation}
Then, using \eqref{z4}, \eqref{z5}, and having in mind that $\mathbb{P}(C=0)+\mathbb{P}(C=1)=1$, we obtain $\mathbb{P}(C=1)=\frac{\lambda}{\mu}$. Then, substituting back in \eqref{z4} we obtain
$p_{0,.}(0)=p_{.,0}(0)=1-\frac{\lambda(\lambda+2\alpha)}{2\alpha\mu}\geq 0$.
\end{proof}
\begin{proposition}\label{prop20}
If $\rho:=\lambda(\lambda+2\alpha)/(2\alpha\mu)=1$, then
$p_{i,j}(0)=p_{i,j}(1)=0$, $i,j=0,1,\ldots$.
\end{proposition}
\begin{proof}
Let $\rho=1$. Thus, $p_{0,.}(0)=p_{.,0}(0)=0$. This means that $p_{0,j}(0)=0$, $j=0,1,\ldots$ (resp. $p_{i,0}(0)=0$, $i=0,1,\ldots$). Therefore, from the first in \eqref{z0} we have $p_{0,j}(1)=0$, $j=0,1,\ldots$ (resp. $p_{i,0}(1)=0$, $i=0,1,\ldots$). We now use an induction argument to show that $p_{i,j}(0)=0$ for $i,j=0,1,\ldots$. We have already prove the assertion for $j=0$. Assume now that $p_{i,j}(0)=0$ for $j=0,1,\ldots,w$. We show that it is also true for $j= w + 1$. From the first in \eqref{z0} and the induction assumption, $p_{i,w}(0)=p_{i,w}(1)$, $i=1,2,\ldots$. The latter equality implies, using the second in \eqref{z0} that $p_{i,w+1}(0)=0$. Thus, $p_{i,j}(0)=0$ for $i=0,1,\ldots$ and $j=w+1$, and completes the induction argument, proving that $p_{i,j}(0)=0$ for $i,j=0,1,\ldots$. It remains to show that $p_{i,j}(1)=0$ for $i,j=0,1,\ldots$. The latter is true for $i,j=1,2,\ldots$ due to the first in \eqref{z0}. It is also true for $j=0$, $i=0,1,\ldots$, as shown previously. It remains to investigate the case for $i=0$, $j=0,1,\ldots$. Using the first in \eqref{z0} for $i=0$, $j=0,1,\ldots$ and the fact that $p_{0,j}(0)=0$ for $j=0,1,\ldots$, we get $p_{0,j}(1)=0$, while we have already noticed that $p_{0,0}(1)=0$. Summarizing, $p_{i,j}(0)=p_{i,j}(1)=0$, $i,j=0,1,\ldots$, so that $\mathbb{P}(Q_{1}=i,Q_{2}=j)=p_{i,j}(0)+p_{i,j}(1)=0$.
\end{proof}\vspace{-0.1in}

Propositions \ref{prop10}, \ref{prop20} imply that the condition $\rho<1$ is necessary for the system to be stable. We will show in Section \ref{sec:compe} that this condition is also sufficient. Assume hereon that $\rho<1$. Note that there is an intuitive explanation of the stability condition: As shown in Proposition \ref{prop10}, in a stable system, $\lambda/\mu$ is the fraction of time the server in the main queue is busy, or equivalently, is the proportion of jobs sent to the orbit queues. Thus, the maximal rate at which jobs flow into orbit queues is $\frac{\lambda}{\mu}(\lambda+2\alpha)$. This rate must be smaller than the corresponding maximal retrial rate, which is $2\alpha$.

In deriving the stationary distribution using CM, it is more convenient to work with the transformed process $X(t)=\{(X_{1}(t),X_{2}(t),C(t)),t\geq0\}$, where $X_{1}(t)=min\{Q_{1}(t),Q_{2}(t)\}$, $X_{2}(t)=|Q_{2}(t)-Q_{1}(t)|$, and state space $S=\{(m,n,k):m,n=0,1,\ldots,k=0,1\}$; see Figure \ref{sss}. Our aim is to determine $q_{m,n}(k)=\lim_{t\to\infty}\mathbb{P}((X_{1}(t),X_{2}(t),C(t))=(m,n,k))$, $(m,n,k)\in S$. Let the column vector $\mathbf{q}(m,n):=(q_{m,n}(0),q_{m,n}(1))^{T}$, where $\mathbf{x}^{T}$ denotes the transpose of vector (or matrix) $\mathbf{x}$. The equilibrium equations are now written as follows:
\small{\begin{eqnarray}
\mathbf{A}_{0,0}\mathbf{q}(0,0)+\mathbf{A}_{0,-1}\mathbf{q}(0,1)=\mathbf{0}\label{e1},\\
\mathbf{B}_{0,0}\mathbf{q}(0,1)+\mathbf{A}_{0,-1}\mathbf{q}(0,2)+2\mathbf{A}_{-1,1}\mathbf{q}(1,0)+\mathbf{A}_{0,1}\mathbf{q}(0,0)=\mathbf{0}\label{e2},\\
\mathbf{B}_{0,0}\mathbf{q}(0,n)+\mathbf{A}_{0,-1}\mathbf{q}(0,n+1)+\mathbf{A}_{-1,1}\mathbf{q}(1,n-1)=\mathbf{0},\,n\geq 2\label{e3}\\
\mathbf{C}_{0,0}\mathbf{q}(m,0)+\mathbf{A}_{0,-1}\mathbf{q}(m,1)+\mathbf{A}_{1,-1}\mathbf{q}(m-1,1)=\mathbf{0},\,m\geq 1\label{e4}\\
\mathbf{C}_{0,0}\mathbf{q}(m,1)+\mathbf{A}_{0,-1}\mathbf{q}(m,2)+2\mathbf{A}_{-1,1}\mathbf{q}(m+1,0)+\mathbf{A}_{1,-1}\mathbf{q}(m-1,2)+\mathbf{A}_{0,1}\mathbf{q}(m,0)=\mathbf{0},\,m\geq 1\label{e5}\\
\mathbf{C}_{0,0}\mathbf{q}(m,n)+\mathbf{A}_{0,-1}\mathbf{q}(m,n+1)+\mathbf{A}_{1,-1}\mathbf{q}(m-1,n+1)+\mathbf{A}_{-1,1}\mathbf{q}(m+1,n-1)=\mathbf{0},\,m\geq 1,n\geq2,\label{e6}
\end{eqnarray}}
where $\mathbf{B}_{0,0}=\mathbf{A}_{0,0}-\mathbf{H}$, $\mathbf{C}_{0,0}=\mathbf{A}_{0,0}-2\mathbf{H}$,\vspace{-0.1in}
\small{\begin{displaymath}
\begin{array}{rl}
\mathbf{A}_{1,-1}=\mathbf{A}_{0,1}=\begin{pmatrix}
0&0\\0&\lambda
\end{pmatrix},&\mathbf{A}_{0,-1}=\mathbf{A}_{-1,1}=\begin{pmatrix}
0&0\\\alpha&0
\end{pmatrix},\,
\mathbf{A}_{0,0}=\begin{pmatrix}
-\lambda&\mu\\\lambda&-(\lambda+\mu)
\end{pmatrix},\,
\mathbf{H}=\begin{pmatrix}
\alpha&0\\0&0
\end{pmatrix}.\vspace{-0.1in}
\end{array}
\end{displaymath}}
%\begin{figure}
%\centering
 % \includegraphics[width=0.5\textwidth]{mmc.pdf}
%\caption{State transition diagram $\{X(t),t\geq0\}$}
%\end{figure}
The transition rate matrix of $\{X(t)\}$ is given by
\begin{displaymath}
\begin{array}{ccc}
\mathbf{Q}=\begin{pmatrix}
\bar{T}_{0}&T_{1}&&&\\
T_{-1}&T_{0}&T_{1}&&\\
&T_{-1}&T_{0}&T_{1}&\\
&&\ddots&\ddots&\ddots\\
\end{pmatrix},&\small{\bar{T}_{0}=\begin{pmatrix}
\mathbf{A}^{T}_{0,0}&\mathbf{A}^{T}_{0,1}&&\\
\mathbf{A}^{T}_{0,-1}&\mathbf{B}^{T}_{0,0}&&\\
&\mathbf{A}^{T}_{0,-1}&\mathbf{B}^{T}_{0,0}&\\
&&\ddots&\ddots\\
\end{pmatrix}},&\small{T_{1}=\begin{pmatrix}
\mathbb{O}&\mathbb{O}&&\\
\mathbf{A}^{T}_{1,-1}&\mathbb{O}&&\\
&\mathbf{A}^{T}_{1,-1}&\mathbb{O}&\\
&&\ddots&\ddots\\
\end{pmatrix}},\\
\small{T_{0}=\begin{pmatrix}
\mathbf{C}^{T}_{0,0}&\mathbf{A}^{T}_{0,1}&&\\
\mathbf{A}^{T}_{0,-1}&\mathbf{C}^{T}_{0,0}&&\\
&\mathbf{A}^{T}_{0,-1}&\mathbf{C}^{T}_{0,0}&\\
&&\ddots&\ddots\\
\end{pmatrix}},&\small{T_{-1}=\begin{pmatrix}
\mathbb{O}&2\mathbf{A}^{T}_{-1,1}&&\\
&\mathbb{O}&\mathbf{A}^{T}_{-1,1}&&\\
&&\mathbb{O}&\mathbf{A}^{T}_{-1,1}&\\
&&&\ddots&\ddots\\
\end{pmatrix},}
\end{array}
\end{displaymath}
so that $\{X(t)\}$ is a quasi birth death (QBD) process with repeated blocks $T_{-1}$, $T_{0}$, $T_{1}$. 
%\section{Stability condition}\label{sec:stabi}
\begin{figure}[h]
\centering
\includegraphics[scale=0.5]{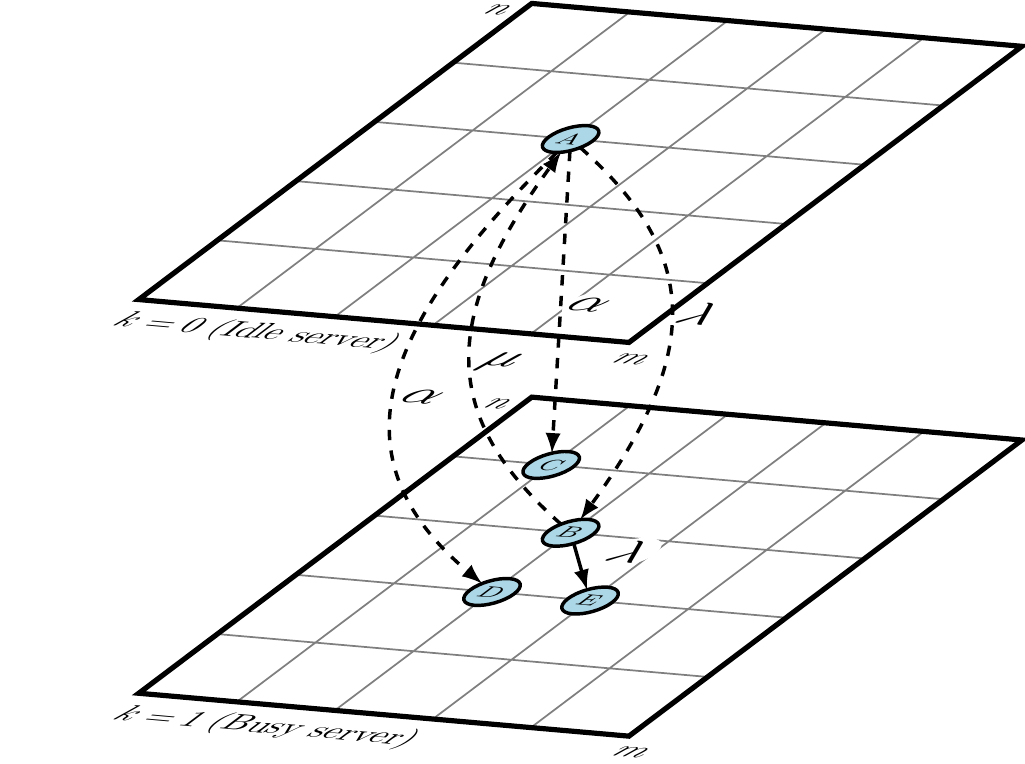}\,\,\,\,\,\,\,\,\includegraphics[scale=0.6]{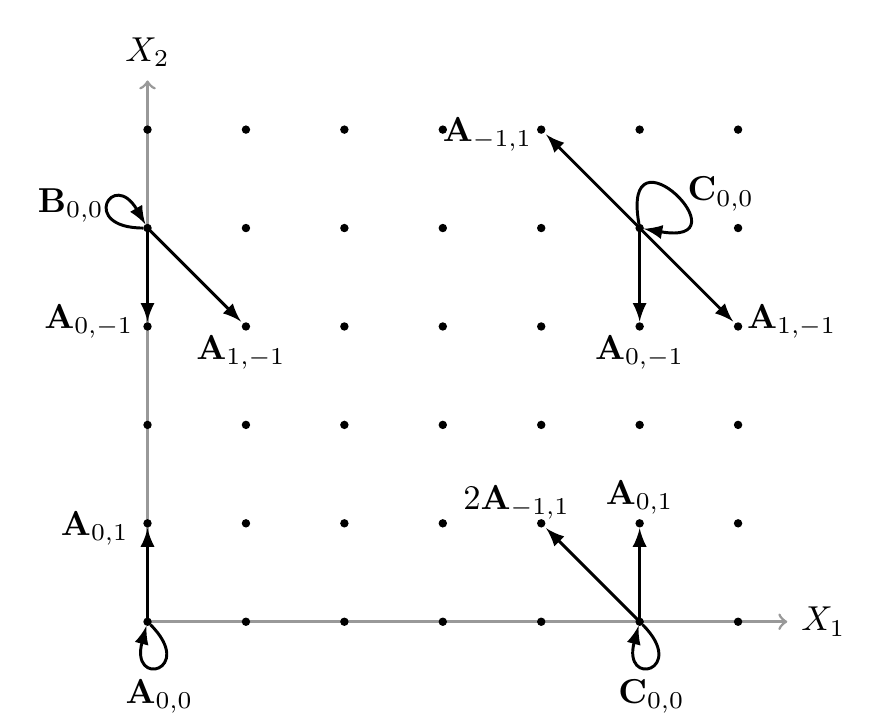}
\caption{Transition diagram of the transformed process.}
\label{sss}
\end{figure}\vspace{-0.2in}
\section{Preliminary results}\label{sec:decay}
In this section we present two results, which provide intuition that the equilibrium distribution $q_{m,n}(k)$ can be written in product form. We first show that the stationary tail probabilities for $min\{Q_{1},Q_{2}\}$ are asymptotically geometric when the difference of the orbit queue sizes and the state of the server are fixed. This result, among other mentioned in the following, indicates that as $m\to\infty$, $q_{m,n}(k)$ would have a product form, and its tail decay rate is $\rho^{2}$. Let $S_{*}=\{0,1,\ldots\}\times\{0,1\}$.
\begin{theorem}\label{dec}
For $\rho<1$, and fixed $(n,k)$, $\lim_{m\to\infty}\rho^{-2m}q_{m,n}(k)=Wp(n,k)$, $W>0$, $\mathbf{p}=\{p(n,k);(n,k)\in S_{*}\}$ as given in Lemma \ref{lem8}. \vspace{-0.1in}
\end{theorem}
%\begin{theorem}\label{dec}
%For $\rho<1$, $0<\delta<1$, we have $\lim_{m\to\infty}\rho^{-2m}q_{m,n}(k)=y_{k}\delta^{n}$, $y_{k}$, $k=0,1$ are positive constants independent of $m$.\vspace{-0.1in}
%\end{theorem}
\begin{proof}
See \ref{appc}.
\end{proof}\vspace{-0.1in}

Motivated by the standard Markovian JSQ model, in which the tail decay rate is the square of the tail decay rate of the M/M/2, Theorem \ref{dec} states that our model has a similar behaviour. More precisely, consider a single server retrial queue with a single orbit of infinite capacity, in which jobs arrive according to a Poisson process with rate $\lambda$, and service times are exponentially distributed with rate $\mu$. Arriving jobs that find the server busy, are routed to the orbit, from which, they retry to access the server according to the \textit{limited} classical retrial policy: If there is only one job in orbit, it retries after an an exponentially distributed time with rate $\alpha$. If there are at least two jobs in orbit, the retrial rate changes to $2\alpha$. Denote by $Q$ the equilibrium orbit queue length of this model, called the \textit{reference model}. Using the matrix geometric method \cite{neuts} we can easily show that the decay rate of this model is $\rho$ (to save space we omit the proof). Theorem \ref{dec}, can be understood by comparing the original model with the \textit{reference model}. Since both models work at full capacity whenever the total number of customers grows, we expect that $\mathbb{P}({Q}_1+{Q}_2=m)$, $\mathbb{P}({Q}=m)$ will have the same decay rate. We also expect that for increasing values of $m$, $\mathbb{P}(min\{Q_{1},Q_{2}\}=m)\simeq\mathbb{P}({Q}_1+{Q}_2=2m)\simeq\mathbb{P}({Q}=2m)$, since the JSQ policy constantly aims at balancing the lengths of the two orbit queues over time. Therefore, Theorem \ref{dec} states that the decay rate of the tail probabilities for $\min\{Q_1,Q_2\}$ for fixed values of the server state and of the difference of the queue lengths equals to the square of the decay rate of the tail probabilities of $Q$. The \textit{reference model} for our JSQ with retrials is the analogue of M/M/2 for the classical JSQ.% leads to the following conjectured behaviour of the minimum orbit queue length: $\mathbb{P}(min\{Q_{1},Q_{2}\}=m,C=k)\simeq y_{k}\rho^{2m}$, as $m\to\infty$, for a positive constant $y_{k}$, $k=1,2$.  %In this work, we obtain exact expressions for $q_{m,n}(k)$, which renders the following result.
%\begin{proposition}
%For $\rho<1$, we have $\lim_{m\to\infty}\rho^{-2k}\mathbb{P}(min\{Q_{1},Q_{2}\}=m)=\frac{y_{0}+y_{1}}{1-\delta}$, for $0<\delta<1$, and $y_{k}$, $k=0,1$ are positive constants independent of $m$.
%\end{proposition}
\begin{figure}[htb]
\centering
\includegraphics[scale=0.5]{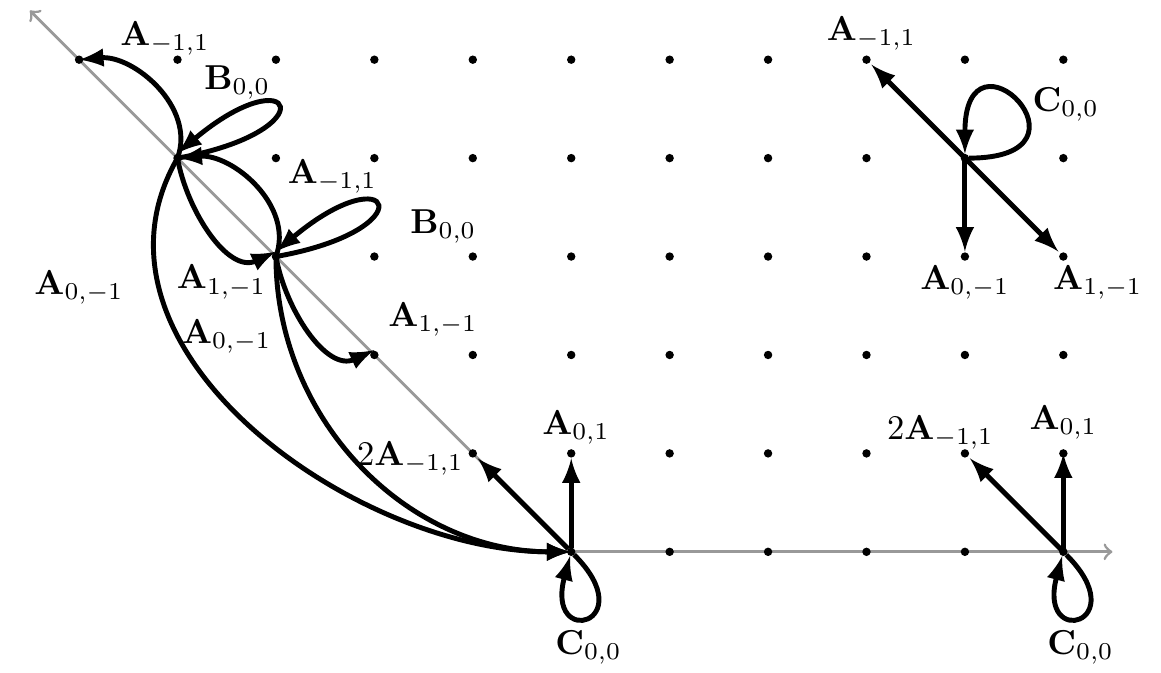}
\caption{Transition diagram of the modified model.}
\label{ssso}
\end{figure}

Next, we conjecture that the inner balance equations have a product form solution. To show this, we exploit the result developed in Theorem \ref{dec} about tail asymptotics. We construct a closely related model that has the same behaviour in the interior as the original model: Starting from the state space of the original model we bend the vertical axis such that the modified model has the same equilibrium equations in the interior and on the horizontal boundary; see Figure \ref{ssso}. \vspace{-0.1in}
\begin{lemma}\label{lem0}
For $\rho<1$, the equilibrium equations \eqref{e4}-\eqref{e6} have a unique up to a constant solution of the form\vspace{-0.15in}
\small{\begin{equation}
\begin{array}{c}
\mathbf{q}(m,n)=\rho^{2m}\mathbf{u}(n),\,m>0,n\geq0,\vspace{-0.15in}
\end{array}
\label{eq_hat_p-geom}
\end{equation}}
with $\mathbf{u}(n)=(u_{0}(n),u_{1}(n))^{T}$ non-zero such that $\sum_{n=0}^{\infty}\rho^{-2n}u_{k}(n)<\infty$, $k=0,1$.
\end{lemma}
\begin{proof}
 We consider a modified model, which is closely related to the original one described by $\tilde{X}(t)$ and that is expected to have the same asymptotic behavior. This modified model is considered on a slightly different grid, namely $\{(m,n,k):m\geq0,n\geq0,k=0,1\}\cup\{(m,n,k):m<0,2m+n\geq 0,k=0,1\}$.
 
 In the interior and on the horizontal boundary, the modified model has the same transition rates as the original model. A characteristic feature of the modified model is that its balance equations for $2m+n = 0$ are exactly the same as the ones in the interior (i.e., the modified model has no ``vertical'' boundary equations) and both models have the same stability region. Therefore, the balance equations for the modified model are given by \eqref{e4}-\eqref{e6} for all $2m+n\geq0$, $m\in \mathbb{Z}$ with only the equation for state $(0,0,k)$, $k=0,1,$ being different due to the incoming rates from the states with $2m+n=0$, $m\in \mathbb{Z}$.
 
 Observe that the modified model, restricted to an area of the form $\{(m,n,k):\, 2m\geq m_0-n,n\geq 0,m_0=1,2,\ldots,k=0,1\}$ embarked by a line parallel to the $2m+n=0$ axis, yields the exact same process. Hence, we can conclude that the equilibrium distribution of the modified model, say $\hat{\mathbf{q}}_{m,n}:=(\hat{q}_{m,n}(0),\hat{q}_{m,n}(1))^{T}$, satisfies
$\hat{\mathbf{q}}_{m+1,n}=\gamma \hat{\mathbf{q}}_{m,n},\ 2m\geq-n,\ n\geq 0$,
 and therefore\vspace{-0.1in}
 \begin{equation}
 \begin{array}{c}
  \hat{\mathbf{q}}_{m,n}=\gamma^m\hat{\mathbf{q}}_{0,n},\ 2m\geq-n,\ n\geq 0.\vspace{-0.1in}
 \end{array}
 \label{eq.geometric_p(m,n)_1}
 \end{equation}
We further observe that $\sum_{n=0}^\infty\hat{q}_{-n,n}(k)=\sum_{n=0}^\infty \gamma^{-n}\hat{q}_{0,n}(h)<1$. To determine the term $\gamma$ we consider levels of the form $(L,k)=\{(m,n,k),:2m+n=L\}$, $k=0,1$ and let $\hat{\mathbf{q}}_L=\sum_{2m+n=L}\hat{\mathbf{q}}_{m,n}$. The balance equations among the levels are:
 \begin{align}
 \mathbf{C}_{0,0}\hat{\mathbf{q}}_{L}+\mathbf{A}_{1,-1}\hat{\mathbf{q}}_{L-1}+2\mathbf{A}_{0,-1}\hat{\mathbf{q}}_{L+1}=0,\ L\geq1,\vspace{-0.1in}\label{26}
 \end{align}
Moreover, equation \eqref{eq.geometric_p(m,n)_1} yields\vspace{-0.1in}
 \begin{equation}
 \begin{array}{c}
 \hat{\mathbf{q}}_{L+1}=\sum_{2k+l=L+1}\gamma^k \hat{\mathbf{q}}_{0,n}=\gamma\sum_{2k+l=L-1}\gamma^k \hat{\mathbf{q}}_{0,n} =\gamma\hat{\mathbf{q}}_{L-1}.\vspace{-0.1in}
 \end{array}\label{31}
 \end{equation}
 Substituting \eqref{31} into \eqref{26} yields
$ \hat{\mathbf{q}}_{L+1}=-[\gamma(\mathbf{A}_{1,-1}+2\gamma\mathbf{A}_{0,-1})^{-1}\mathbf{C}_{0,0}]\hat{\mathbf{q}}_{L}$.
 Combining \eqref{31} with \eqref{26} with $\gamma=\rho^{2}$ yields
\begin{equation*}
\begin{array}{c}
det(\rho\mathbf{C}_{0,0}+\mathbf{A}_{1,-1}+\rho^{2}2\mathbf{A}_{0,-1})=0\Leftrightarrow 2\alpha\mu(1-\rho)(\rho-\frac{\lambda(\lambda+2\alpha)}{2\alpha\mu})=0,\vspace{-0.1in}
\end{array}
\end{equation*}
which implies that indeed $\gamma=\rho^{2}$.
 
Thus, it is shown that the equilibrium distribution of the modified model has a product-form solution which is unique up to a positive multiplicative constant. Returning to the original process $X(t)$, we immediately assume that the solution of \eqref{e4}-\eqref{e6} is identical to the expression for the modified model as given in \eqref{eq_hat_p-geom}. Moreover, the above analysis implies that this product-form is unique, since the equilibrium distribution of the modified model is unique. 
\end{proof}\vspace{-0.3in}
\section{The compensation method}\label{sec:compe}
The CM attempts to solve the equilibrium equations by linear combination of product forms. To cope with this task, we first characterize a sufficiently rich basis (that contains uncountably many elements) of product form solutions that satisfy the equilibrium equations \eqref{e6}. Then, this basis is used to construct a linear combination that also satisfies the equilibrium equations corresponding to the boundary states. CM appropriately selects the right elements, and consists of adding on terms so as to compensate alternately for the error on the vertical boundary, and for the error on the horizontal boundary. More details are given in the following subsections. \vspace{-0.15in}
\subsection{Step 1: Construction of the basis of product forms}
Characterize the set of product forms $\gamma^{m}\delta^{n}\boldsymbol{\theta}$ that satisfy the equilibrium equations at the interior of the state space \eqref{e6}. This is presented in Lemmas \ref{lem1}, \ref{lem2}.\vspace{-0.1in}
\begin{lemma}\label{lem1}
The product form $\mathbf{q}(m,n)=\gamma^{m}\delta^{n}\boldsymbol{\theta}$, $m\geq0$, $n\geq1$, $\boldsymbol{\theta}:=(\theta_{0},\theta_{1})^{T}$, with $\theta_{1}=\theta_{0}(\lambda+2\alpha)/\mu$, satisfies \eqref{e6} if\vspace{-0.1in}
\begin{eqnarray}
\begin{array}{rl}
\mathbf{D}(\gamma,\delta)\boldsymbol{\theta}=\boldsymbol{0}\Rightarrow(\gamma\delta\mathbf{C}_{0,0}+\gamma\delta^{2}\mathbf{A}_{0,-1}+\gamma^{2}\mathbf{A}_{-1,1}+\delta^{2}\mathbf{A}_{1,-1})\boldsymbol{\theta}=&\boldsymbol{0},\label{t1}\vspace{-0.1in}
%\boldsymbol{\theta}=&\theta_{0}(1,\frac{\lambda+2\alpha}{\mu})^{T}.\vspace{-0.1in}\label{the}
\end{array}
\end{eqnarray}
\end{lemma}
\begin{proof}
The desired result is obtained directly by substituting the product form in \eqref{e6}. For $det(\mathbf{D}(\gamma,\delta))=0$, the matrix $\mathbf{D}(\gamma,\delta)$ has rank equal to 1. From the system of linear equations $\mathbf{D}(\gamma,\delta)\boldsymbol{\theta}=\boldsymbol{0}$, and the form of $\mathbf{D}(\gamma,\delta)$ we derive the relation among $\theta_{0}$, $\theta_{1}$.
\end{proof}

Note that the value of eigenvector $\boldsymbol{\theta}$ is independent of the values of $\gamma$, $\delta$ that satisfy \eqref{t1}. This result simplifies considerably the analysis.
%\begin{figure}[htp!]
%\centering
%\includegraphics[scale=0.4]{zerosc.pdf}
%\caption{$det(\mathbf{D}(\gamma,\delta))=0$ in $%\mathbb{R}_{+}^{2}$ for $\rho=0.5.$}\label{fig1}
%\end{figure}
The next step is to determine $\gamma$s and $\delta$s such that $0<|\gamma|<1$, $0<|\delta|<1$ for which there exists a non-zero solution $\boldsymbol{\theta}$ of \eqref{t1}, i.e., $det(\mathbf{D}(\gamma,\delta))=0$. The next lemma gives information about the location of the zeros of $det(\mathbf{D}(\gamma,\delta))=0$.
\begin{lemma}\label{lem2}
$(i)$ For $\rho<1$, and for every fixed $\gamma$ with $|\gamma|\in(0,1)$, the equation $det(\mathbf{D}(\gamma,\delta))=0$ takes the form
\begin{equation}
\begin{array}{c}
\gamma\delta 2(\rho+1)-2\rho\delta^2-\gamma^{2}-\gamma\delta^{2}=0,\vspace{-0.1in}
\end{array}
\label{t2}
\end{equation}
and has exactly one root in the $\delta-$plane such that $0<|\delta|<|\gamma|$. $(ii)$ For $\rho<1$, and for every fixed $\delta$ with $|\delta|\in(0,1)$, the equation (\ref{t2}) has exactly one root in the $\gamma-$plane such that $0<|\gamma|<|\delta|$.\vspace{-0.1in}
\end{lemma}
\begin{proof}
Starting from $det(\mathbf{D}(\gamma,\delta))=0$, we obtain \eqref{t2}. Surprisingly, the form of \eqref{t2} is exactly the same as the one of equation (8) in Lemma 1 of the seminal paper \cite{ad0}, and thus, the assertions $i)$, $ii)$ can be proven similarly, so further details are omitted.
\end{proof}\vspace{-0.2in}
\subsection{Step 2: Construction of the formal solution}
We now focus on constructing a formal solution to the equilibrium equations as a linear combination of the elements of the rich basis obtained in Step 1. This is composed in four steps as follows.\vspace{-0.1in}
\paragraph{Step 2a) Initial solution:} The construction of the starting solution is crucial, and starts with a suitable initial term satisfying the interior of the state space and the horizontal boundary equilibrium equations. In Lemma \ref{lem0}, we shown that $\gamma_{0}=\rho^{2}$, and from \eqref{t2} we obtain $\delta_{0}=\frac{\rho^{2}}{2+\rho}$, such that $|\delta_{0}|<|\gamma_{0}|$. Note that this pair is the unique feasible starting pair satisfying the horizontal boundary and the interior equilibrium equations. Moreover, such a behaviour it is also shown in Lemma \ref{lem0}. It is easy to see that no feasible starting pair that satisfies the interior and the vertical boundary equilibrium equations do exists. In Lemma \ref{plm} we specify the form of the vector $\mathbf{u}(n)$. Its proof follows by the substitution of \eqref{init} in \eqref{e4}-\eqref{e6}.\vspace{-0.1in}
\begin{lemma}\label{plm}
The solution
\begin{eqnarray}
\mathbf{q}(m,n)=&\left\{\begin{array}{ll}
h_{0}\gamma_{0}^{m}\delta_{0}^{n}\boldsymbol{\theta},&m\geq0,n\geq1,\\
\gamma_{0}^{m}\boldsymbol{\xi},&m\geq1,n=0,
\end{array}\right.\label{init}\\
\boldsymbol{\xi}=&-\frac{h_{0}}{\gamma_{0}}\mathbf{C}_{0,0}^{-1}[\mathbf{A}_{1,-1}+\gamma_{0}\mathbf{A}_{0,-1}]\delta_{0}\boldsymbol{\theta},\vspace{-0.2in}\label{xi}
\end{eqnarray}
satisfies the equilibrium equations \eqref{e4}-\eqref{e6}. \vspace{-0.1in}
\end{lemma}
\paragraph{Step 2b) Compensation on the vertical boundary:} It is easily seen that the solution in \eqref{init} does not satisfy the vertical boundary equation \eqref{e3}. To compensate for this error we add a new term such that the sum of two terms satisfies \eqref{e3}, \eqref{e6}. Thus, we assume that $h_{0}\gamma_{0}^{m}\delta_{0}^{n}\boldsymbol{\theta}+c_{1}\gamma^{m}\delta^{n}\boldsymbol{\theta}$ satisfies both \eqref{e3}, \eqref{e6}. Substituting this form in \eqref{e3} yields
\begin{equation}
\begin{array}{c}
[h_{0}V(\gamma_{0},\delta_{0}+c_{1}V(\gamma,\delta)]\boldsymbol{\theta}=\mathbf{0},\,n\geq2,\vspace{-0.1in}
\end{array}
\label{xc}
\end{equation}
where $V(\gamma,\delta)=\mathbf{B}_{0,0}\delta+\mathbf{A}_{-1,1}\gamma+\mathbf{A}_{0,-1}\delta^{2}$. Hence, $\delta=\delta_{0}$, and for such $\delta_{0}$, we obtain from \eqref{t1} $\gamma:=\gamma_{1}$, such that $|\gamma_{1}|<|\delta_{0}|<|\gamma_{0}|$, so that $(\gamma_{1},\delta_{0})$ satisfies \eqref{e3}. Thus, the solution $h_{0}\gamma_{0}^{m}\delta_{0}^{n}\boldsymbol{\theta}+c_{1}\gamma_{1}^{m}\delta_{0}^{n}\boldsymbol{\theta}$ satisfies \eqref{e3}. The constant $c_{1}$ can be obtained by substituting it in \eqref{e3}, or equivalently, by using \eqref{xc} and the fact that $\gamma$, $\delta$ are the roots of \eqref{t2}. This procedure yields after some algebra
\begin{equation}
\begin{array}{c}
c_{1}=-\frac{\gamma_{1}-\delta_{0}\left(\frac{\lambda+\mu}{\mu}\right)}{\gamma_{0}-\delta_{0}\left(\frac{\lambda+\mu}{\mu}\right)}h_{0}.\vspace{-0.1in}
\end{array}
\label{c}
\end{equation}
%
 %as follows. Use \eqref{t1} to see that 
%\begin{displaymath}
%V(\gamma,\delta)\boldsymbol{\theta}=\delta[\mathbf{B}_{0,0}-\mathbf{C}_{0,0}-\frac{\delta}{\gamma}\mathbf{A}_{1,-1}]\boldsymbol{\theta}.
%\end{displaymath}
%Substituting back in \eqref{xc} yields 
\paragraph{Step 2c) Compensation on the horizontal boundary:} Adding the new term, we violate the horizontal boundary equations \eqref{e4}, \eqref{e5}. We compensate for this error by adding a product form generated by the pair $(\gamma_{1},\delta_{1})$, such that $|\delta_{1}|<|\gamma_{1}|$. The new solution is:
\small{\begin{equation}
\begin{array}{c}
\mathbf{q}(m,n)=\left\{\begin{array}{ll}
h_{0}\gamma_{0}^{m}\delta_{0}^{n}\boldsymbol{\theta}+c_{1}\gamma_{1}^{m}\delta_{0}^{n}\boldsymbol{\theta}+h_{1}\gamma_{1}^{m}\delta_{1}^{n}\boldsymbol{\theta},&m,n\geq1,\\
\gamma_{0}^{m}\boldsymbol{\xi}+\gamma_{1}^{m}\boldsymbol{\xi}_{1},&m\geq1,n=0,
\end{array}\right.\vspace{-0.1in}
\end{array}
\label{hc}
\end{equation}}
where $h_{1}$, $\boldsymbol{\xi}_{1}$ are obtained such that to satisfy \eqref{e4}-\eqref{e6}. In particular, by substituting \eqref{hc} to \eqref{e4} yields
\small{\begin{equation}
\begin{array}{c}
\boldsymbol{\xi}_{1}=-\frac{1}{\gamma_{1}}[\gamma_{0}\boldsymbol{\xi}+\mathbf{C}_{0,0}^{-1}(\mathbf{A}_{1,-1}+\gamma_{1}\mathbf{A}_{0,-1})(c_{1}\delta_{0}+h_{1}\delta_{1})\boldsymbol{\theta}].
\end{array}
\label{wxi}
\end{equation}}
Now, note that \eqref{e4} reads $\mathbf{q}(m,0)=-\mathbf{C}_{0,0}^{-1}[\mathbf{A}_{1,-1}\mathbf{q}(m-1,1)+\mathbf{A}_{0,-1}\mathbf{q}(m,1)]$, $m\geq1$. Substituting back to \eqref{e5} yields
\small{\begin{equation}
\begin{array}{l}
[\mathbf{C}_{0,0}-(\mathbf{A}_{0,1}\mathbf{C}_{0,0}^{-1}\mathbf{A}_{0,-1}+2\mathbf{A}_{-1,1}\mathbf{C}_{0,0}^{-1}\mathbf{A}_{1,-1})]\mathbf{q}(m,1)+\mathbf{A}_{1,-1}\mathbf{q}(m-1,2)+\mathbf{A}_{0,-1}\mathbf{q}(m,2)\\
-\mathbf{A}_{0,1}\mathbf{C}_{0,0}^{-1}\mathbf{A}_{1,-1}\mathbf{q}(m-1,1)-2\mathbf{A}_{-1,1}\mathbf{C}_{0,0}^{-1}\mathbf{A}_{0,-1}\mathbf{q}(m+1,1)=\mathbf{0},\,m\geq1.
\end{array}\label{hc2}
\end{equation}}
Substituting \eqref{hc} in \eqref{hc2} yields after tedious algebra that $h_{1}$ should satisfy
\begin{equation}
\begin{array}{l}
[h_{1}L(\gamma_{1},\delta_{1})+c_{1}L(\gamma_{1},\delta_{0})]\boldsymbol{\theta}=\mathbf{0},\\
L(\gamma,\delta):=\delta\left[\gamma[\mathbf{C}_{0,0}-(\mathbf{A}_{0,1}\mathbf{C}_{0,0}^{-1}\mathbf{A}_{0,-1}+2\mathbf{A}_{-1,1}\mathbf{C}_{0,0}^{-1}\mathbf{A}_{1,-1})]\right.\\
\left.+\delta\mathbf{A}_{1,-1}+\gamma\delta\mathbf{A}_{0,-1}-\delta\mathbf{A}_{0,1}\mathbf{C}_{0,0}^{-1}\mathbf{A}_{1,-1}-\gamma^{2}2\mathbf{A}_{-1,1}\mathbf{C}_{0,0}^{-1}\mathbf{A}_{0,-1}\right].\end{array}
\label{hor2}
\end{equation}

Thus, \eqref{hor2} implies that $det(h_{1}L(\gamma_{1},\delta_{1})+c_{1}L(\gamma_{1},\delta_{0}))=0$, and having in mind \eqref{t2}, we obtain after some algebra that
\small{\begin{equation}
\begin{array}{c}
h_{1}=-\frac{(\rho+\gamma_{1})/\delta_{1}-(1+\rho)}{(\rho+\gamma_{1})/\delta_{0}-(1+\rho)}c_{1}.
\end{array}
\label{h}
\end{equation}}\vspace{-0.25in}
%\begin{displaymath}
%\begin{array}{l}
%h_{1}[\mathbf{S}(\gamma_{1})\mathbf{C}_{0,0}^{-1}(\mathbf{A}_{1,-1}+\gamma_{1}\mathbf{A}_{0,-1})\delta_{1}+\mathbf{A}_{-1,1}\gamma_{1}^{2}]\boldsymbol{\theta}\\
%=\gamma_{0}[\mathbf{S}(\gamma_{0})-\mathbf{S}(\gamma_{1})\boldsymbol{\xi}]-c[\mathbf{S}(\gamma_{1})\mathbf{C}_{0,0}^{-1}(\mathbf{A}_{1,-1}+\gamma_{1}\mathbf{A}_{0,-1})\delta_{0}+\mathbf{A}_{-1,1}\gamma_{1}^{2}]\boldsymbol{\theta},
%\end{array}
%\end{displaymath}
%where $\mathbf{S}(\gamma):=\mathbf{A}_{0,1}+2\mathbf{A}_{-1,1}\gamma$.
\paragraph{Step 2d) Formal construction of the solution:} By constantly repeating steps 2b), 2c) as above, we construct the entire formal series as shown below.\vspace{-0.1in}
%\small{\begin{equation}
%\begin{array}{lr}
%\mathbf{q}(m,n)=\left\{\begin{array}{ll}
%\sum_{i=0}^{\infty}(h_{i}\gamma_{i}^{m}+c_{i+1}\gamma_{i+1}^{m})\delta_{i}^{n}\boldsymbol{\theta},&m\geq0,n\geq1,\\
%\gamma_{0}^{m}\boldsymbol{\xi}+\sum_{i=1}^{\infty}\gamma_{i}^{m}\boldsymbol{\xi}_{i},&m%\geq1,n=0,
%\end{array}\right.,&\,\,\mathbf{q}(0,0)=-\mathbf{A}_{0,0}^{-1}\mathbf{A}_{0,-1}\mathbf{q}(0,1).\vspace{-0.1in}
%\end{array}
%\label{for}
%\end{equation}}
\begin{theorem}\label{theorem}
For $\rho<1$,
\small{\begin{equation}
\begin{array}{rl}
\mathbf{q}(m,n)\propto&\sum_{i=0}^{\infty}(h_{i}\gamma_{i}^{m}+c_{i+1}\gamma_{i+1}^{m})\delta_{i}^{n}\boldsymbol{\theta},\,\,({\text{pairs with the same $\delta$-term})},\,m\geq0,n\geq1,\\
\propto	&(h_{0}\gamma_{0}^{m}\delta_{0}^{n}+\sum_{i=0}^{\infty}(h_{i+1}\delta_{i}^{n}+c_{i+1}\delta_{i+1}^{n})\gamma_{i+1}^{m}\boldsymbol{\theta},\,\,({\text{pairs with the same $\gamma$-term})}\,m\geq0,n\geq1,\\
\mathbf{q}(m,0)\propto&(\gamma_{0}^{m}\boldsymbol{\xi}+\sum_{i=1}^{\infty}\gamma_{i}^{m}\boldsymbol{\xi}_{i}),\,m\geq1,\vspace{-0.1in}
\end{array}
\label{form}
\end{equation}}
and $\mathbf{q}(0,0)=-\mathbf{A}_{0,0}^{-1}\mathbf{A}_{0,-1}\mathbf{q}(0,1)$, where the symbol ($\propto$) means ``directly proportional". Moreover, $\boldsymbol{\theta}=\theta_{0}(1,\frac{\lambda+2\alpha}{\mu})^{T}$, $\theta_{0}>0$, and the sequences $\{\gamma_{i}\}_{i\in\mathbb{N}}$, $\{\delta_{i}\}_{i\in\mathbb{N}}$, $\{h_{i}\}_{i\in\mathbb{N}}$, $\{c_{i}\}_{i\in\mathbb{N}}$, $\{\boldsymbol{\xi}_{i}\}_{i\in\mathbb{N}}$, are
obtained recursively based on the analysis above.
\end{theorem}\vspace{-0.25in}
\subsection{Step 3: Absolute convergence of the solution \& normalization constant}
We now have to show that the solutions \eqref{form} converge in two steps: $i)$ to show that the sequences $\{\gamma_{i}\}_{i\in\mathbb{N}}$, $\{\delta_{i}\}_{i\in\mathbb{N}}$ converge to zero exponentially fast (Proposition \ref{prop1}, Lemmas \ref{lemp}, \ref{lemp1}), and $ii)$ the formal solution converges absolutely (Proposition \ref{pr11}).
%\begin{remark}
%Note that since $\boldsymbol{\theta}$ is independent of %$\gamma$, $\delta$, the series in \eqref{form} for $m\geq0$, $n\geq1$ satisfy the following balance property:
%$(\lambda+2\alpha)q_{m,n}(0)=\mu q_{m,n}(1),\,m\geq0,\,n\geq1$.
%\end{remark}
To show that, we need some preliminary results. Since \eqref{t2} has the same form as in \cite[eq. (8)]{ad0}, the sequences $\{\gamma_{i}\}_{i\in\mathbb{N}}$, $\{\delta_{i}\}_{i\in\mathbb{N}}$ satisfy $1>\rho^{2}=|\gamma_{0}|>\frac{\rho^2}{2+\rho}=|\delta_{0}|>|\gamma_{1}|>|\delta_{1}|>\ldots$.\vspace{-0.1in}   %The preliminary results are easily proven as those presented in \cite{ad0} due to the fact that \textit{fundamental equation} for generating the sequences $\{\gamma_{i}\}_{i\in\mathbb{N}}$, $\{\delta_{i}\}_{i\in\mathbb{N}}$ has exactly the same form as the one in \eqref{t2}.
\begin{proposition}\label{prop1}
The sequences $\{\gamma_{i}\}_{i\in\mathbb{N}}$, $\{\delta_{i}\}_{i\in\mathbb{N}}$ in \eqref{form} satisfy: $0\leq|\gamma_{i}|\leq (\frac{1}{3})^{i}\rho^{2}$, and $0\leq|\delta_{i}|\leq (\frac{1}{3})^{i+1}\rho^{2}$.
\end{proposition}
\begin{proof}
%Note that $\gamma_{0}=\rho^2<1$ and $\delta_{0}$ follows from \eqref{t2} according to Lemma \ref{lem2} with $|\gamma_{0}|>|\delta_{0}|$.
%Thus, assertion $1.$ follows upon repeating this argument and in light of Lemma \ref{lem2}. Thus, by using \eqref{t2}, each $\gamma_{i}$ generates a $\delta_{i}$, such that $|\delta_{i}|<|\gamma_{i}|$, and each $\delta_{i}$ generates a $\gamma_{i+1}$, such that $|\gamma_{i+1}|<|\delta_{i}|$. As a consequence,
%\begin{displaymath}
%\begin{array}{c}
%1=\rho^{2}=|\gamma_{0}|>\frac{\rho^2}{2+\rho}=|\delta_{0}|>|\gamma_{1}|>|\delta_{1}|>\ldots
%\end{array}
%\end{displaymath}
%Besides, $\gamma_{i}\gamma_{i+1}=2\rho\delta_{i}^{2},\gamma_{i}+\gamma_{i+1}=2(1+\rho)\delta_{i}-\delta_{i}^{2},\,
%\delta_{i}\delta_{i+1}=\frac{\gamma_{i+1}^{2}}{2\rho+\gamma_{i+1}},\delta_{i}+\delta_{i+1}=\frac{2(1+\rho)}{2\rho+\gamma_{i+1}}\gamma_{i+1}$.

%Since $\gamma_{0}$, $\delta_{0}$ are positive, it follows from \eqref{lp} by induction that all $\gamma_{i}$, $\delta_{i}$ are positive.
We first show that $a)$ for a fixed $\gamma$, with $|\gamma|<\gamma_{0}$, $|\delta|<\frac{|\gamma|}{2}$, and then, $b)$ for a fixed $\delta$, with $|\delta|\leq\gamma_{0}/3$, we have $|\gamma|<\frac{2}{3}|\delta|$.\\
$a)$ For a fixed $\gamma$, set $z=\delta/\gamma$ on $|z|=1/2$. Under this transform, \eqref{t2} reads $0=(2\rho+\gamma)z^{2}-2(1+\rho)z+1$. Set $f(z):=2(1+\rho)z$, $g(z)=(2\rho+\gamma)z^{2}+1$. Note that $|f(z)|=2(1+\rho)|z|=1+\rho$, and $|g(z)|=|(2\rho+|\gamma|)z^{2}+1|\leq (2\rho+|\gamma|)|z|^{2}+1=(2\rho+|\gamma|)\frac{1}{4}+1$.
Moreover, $(2\rho+|\gamma|)\frac{1}{4}+1<\rho+1\Leftrightarrow |\gamma|<2\rho$. Note that $|\gamma|\leq\gamma_{0}=\rho^{2}<2\rho$. Thus, since $|g(z)|<|f(z)|$ on $|z|=1/2$, Rouch\'e's theorem \cite{titc} completes the proof of $a)$.\\
$b)$ For a fixed $\delta$, we show that $|\gamma|<\frac{2}{3}|\delta|$, by setting now $w=\gamma/\delta$ in \eqref{t2} on the domain $|w|=2/3$. Then, \eqref{t2} reads,
$w^{2}+(\delta-2(1+\rho))w+2\rho=0$. Set $h(w):=w^{2}$, $m(w):=w(2(1+\rho)-\delta)-2\rho$. Note that $m(w)$ has a single zero in the interior of $w=2/3$. Then, $|h(w)|=4/9$ and $|m(w)|\geq\frac{2}{3}|2(1+\rho)-|\delta|-3\rho|=\frac{2}{3}|2-\rho-|\delta||=\frac{2}{3}(2-\rho-|\delta|)$. Note that, $\frac{2}{3}<2-\rho-|\delta|\Leftrightarrow |\delta|<\frac{4-3\rho}{3},\text{ and that, }|\delta|\leq\frac{\gamma_{0}}{3}=\frac{\rho^{2}}{3}<\frac{4-3\rho}{3}$. Thus, $|h(w)|<|m(w)|$ on $|w|=2/3$, and Rouch\'e's theorem \cite{titc} completes the proof of $b)$. Then applying $a)$, $b)$ iteratively yields,
\begin{displaymath}
\small{\begin{array}{l}
|\gamma_{i}|\leq\frac{2}{3}|\delta_{i-1}|\leq \frac{2}{3}\frac{1}{2}|\gamma_{i-1}|\leq\ldots\leq(\frac{2}{3}\frac{1}{2})^{i}|\gamma_{0}|=(\frac{1}{3})^{i}\rho^{2},\,
|\delta_{i}|\leq\frac{1}{2}|\gamma_{i}|\leq \frac{1}{2}\frac{2}{3}|\delta_{i-1}|\leq\ldots\leq(\frac{2}{3}\frac{1}{2})^{i}|\delta_{0}|=(\frac{1}{3})^{i}\frac{\rho^{2}}{3}=(\frac{1}{3})^{i+1}\rho^{2}.
\end{array}}
\end{displaymath}
\end{proof}
Proposition \ref{prop1} states that $\gamma_{i}\to0$, $\delta_{i}\to 0$ as $i\to\infty$. In the following, we focus on the asymptotic behaviour of $\delta_{i}/\gamma_{i}$ and $\gamma_{i+1}/\delta_{i}$. This result is important to investigate the convergence of the series in \eqref{form}. 
\begin{lemma}\label{lemp}
$a)$ Let $\gamma_{i}$ fixed and $\delta_{i}$ the root of \eqref{t2} such that $\delta_{i}<\gamma_{i}$. As $i\to\infty$, then $\delta_{i}/\gamma_{i}\to w^{-}$, $|w^{-}|\in (0,1)$ is the smallest root of 
\small{\begin{equation}
\begin{array}{c}
2\rho w^{2}-2(1+\rho)w+1=0.
\end{array}
\label{mmm}
\end{equation} }
$b)$ Let $\delta_{i}$ fixed and $\gamma_{i+1}$ the root in \eqref{t2} such that $\gamma_{i+1}<\delta_{i}$. As $i\to\infty$, then $\gamma_{i+1}/\delta_{i}\to 1/w^{+}$, with $w^{+}>1$ the larger root of \eqref{mmm}.
\end{lemma}
\begin{proof}
See \cite{ad0}, since \eqref{t2} that generates the sequences $\{\gamma_{i}\}_{i\in\mathbb{N}}$, $\{\delta_{i}\}_{i\in\mathbb{N}}$ has the same form as in \cite[eq. (8)]{ad0}.
%$a)$ Divide \eqref{t2} with $\gamma_{i}^{2}$, set $w=\delta_{i}/\gamma_{i}$ and let $i\to\infty$ yields \eqref{mmm}. Since $w^{-}w^{+}=1/2\rho$, and since for $\rho<1$,
%\begin{displaymath}
%\begin{array}{c}
%w^{+}=\frac{1+\rho+\sqrt{1+\rho^{2}}}{2\rho}>\frac{1+\rho+\sqrt{(1-\rho)^{2}}}{2\rho}>\frac{1}{2\rho},
%\end{array}
%\end{displaymath} 
%we have $\delta_{i}/\gamma_{i}\to w^{-}<1$.\\
%$b)$ Divide \eqref{t2} with $\delta_{i}^{2}$, set $\phi=\gamma_{i+1}/\delta_{i}$ and let $i\to\infty$ yields $\phi^{2}-2(1+\rho)\phi+2\rho=0$. We focus on its smallest root, say $\phi^{-}$, which should be smaller than 1. Note that $\phi^{-}=1/w^{+}$, where $w^{+}$ is the larger root of \eqref{mmm}.
\end{proof}

The final ingredient to check the convergence of the series \eqref{form} is to determine the values of the ratios $c_{i+1}/h_{i}$, $h_{i}/c_{i}$ as $i\to\infty$.
\begin{lemma} \label{lemp1}
\begin{enumerate}
\item Let $\gamma_{i}$, $\delta_{i}$, $\gamma_{i+1}$ be roots of \eqref{t2} such that $1>|\gamma_{i}|>|\delta_{i}|>|\gamma_{i+1}|$. Then, as $i\to\infty$, $\frac{c_{i+1}}{h_{i}}\to\frac{\frac{\lambda+\mu}{2\mu\rho}-w^{-}}{w^{+}-\frac{\lambda+\mu}{2\mu\rho}}$.
\item Let $\delta_{i-1}$, $\gamma_{i}$, $\delta_{i}$ be roots of \eqref{t2} such that $1>|\delta_{i-1}|>|\gamma_{i}|>|\delta_{i}|$. Then, as $i\to\infty$, $\frac{h_{i}}{c_{i}}\to-\frac{w^{+}}{w^{-}}$.
\item As $i\to\infty$, $\boldsymbol{\xi}\to-h_{0}\mathbf{C}_{0,0}^{-1}w^{-}\mathbf{A}_{1,-1}\boldsymbol{\theta}$.
\item For $i\geq1$, and $\boldsymbol{\xi}_{0}:=\boldsymbol{\xi}$, the vector $\boldsymbol{\xi}_{i}$,  is such that $\boldsymbol{\xi}_{i}\to-[\frac{w^{+}}{w^{-}}\boldsymbol{\xi}_{i-1}+\mathbf{C}_{0,0}^{-1}\mathbf{A}_{1,-1}[h_{i}w^{-}+w^{+}c_{i})]\boldsymbol{\theta}$, as $i\to\infty$.
\end{enumerate}
\end{lemma}
\begin{proof}
$1.$ Using the indexing of the compensation parameters \eqref{c},
\begin{displaymath}
\begin{array}{c}
\frac{c_{i+1}}{h_{i}}=-\frac{\gamma_{i+1}-\delta_{i}\left(\frac{\lambda+\mu}{\mu}\right)}{\gamma_{i+1}-\delta_{i}\left(\frac{\lambda+\mu}{\mu}\right)}=\frac{\left(\frac{\lambda+\mu}{\mu}\right)-\frac{\gamma_{i+1}}{\delta_{i}}}{\frac{\gamma_{i}}{\delta_{i}}-\left(\frac{\lambda+\mu}{\mu}\right)}\to\frac{\left(\frac{\lambda+\mu}{\mu}\right)-\frac{1}{w^{+}}}{\frac{1}{w^{-}}-\left(\frac{\lambda+\mu}{\mu}\right)}=\frac{w^{+}w^{-}\left(\frac{\lambda+\mu}{\mu}\right)-w^{-}}{w^{+}-w^{+}w^{-}\left(\frac{\lambda+\mu}{\mu}\right)}=\frac{\left(\frac{\lambda+\mu}{2\rho\mu}\right)-w^{-}}{w^{+}-\left(\frac{\lambda+\mu}{2\rho\mu}\right)},
\end{array}
\end{displaymath}
since as $i\to\infty$, Lemma \ref{lemp} implies that $\gamma_{i+1}/\delta_{i}\to 1/w^{+}$, $\gamma_{i}/\delta_{i}\to 1/w^{-}$ , and where the last equality follows from $w^{-}w^{+}=1/2\rho$.\\
$2.$ Similarly, from \eqref{h}, we have $
\frac{h_{i}}{c_{i}}=-\frac{(\rho+\gamma_{i})/\delta_{i}-(1+\rho)}{(\rho+\gamma_{i})/\delta_{i-1}-(1+\rho)}=-\frac{(\rho+\gamma_{i})\delta_{i-1}/\delta_{i}-(1+\rho)\delta_{i-1}}{(\rho+\gamma_{i})-(1+\rho)\delta_{i-1}}$.

Since as $i\to\infty$, $\gamma_{i}\to 0$, $\delta_{i-1}\to0$, and $\frac{\delta_{i-1}}{\delta_{i}}=\frac{\delta_{i-1}}{\gamma_{i}}\frac{\gamma_{i}}{\delta_{i}}\to w^{+}\frac{1}{w^{-}}$, the assertion $2$ is now proved.\\
$3.$ Note that \eqref{xi} and Lemma \ref{lemp} imply
$\boldsymbol{\xi}=-h_{0}\mathbf{C}_{0,0}^{-1}[\mathbf{A}_{1,-1}+\gamma_{0}\mathbf{A}_{0,-1}]\frac{\delta_{0}}{\gamma_{0}}\boldsymbol{\theta}\to-h_{0}\mathbf{C}_{0,0}^{-1}\mathbf{A}_{1,-1}w^{-}\boldsymbol{\theta}$.\\
$4.$ The indexing in \eqref{wxi} implies for $i\geq1$ that
\small{\begin{displaymath}
\begin{array}{rl}
\boldsymbol{\xi}_{i}=&-\frac{1}{\gamma_{i}}[\gamma_{i-1}\boldsymbol{\xi}_{i-1}+\mathbf{C}_{0,0}^{-1}(\mathbf{A}_{1,-1}+\gamma_{i}\mathbf{A}_{0,-1})(c_{i}\delta_{i-1}+h_{i}\delta_{i})\boldsymbol{\theta}]=-[\frac{\gamma_{i-1}}{\delta_{i}}\frac{\delta_{i}}{\gamma_{i}}\boldsymbol{\xi}_{i-1}+\mathbf{C}_{0,0}^{-1}(\mathbf{A}_{1,-1}+\gamma_{i}\mathbf{A}_{0,-1})(c_{i}\frac{\delta_{i-1}}{\gamma_{i}}+h_{i}\frac{\delta_{i}}{\gamma_{i}})\boldsymbol{\theta}]\\
\to&-[\frac{w^{+}}{w^{-}}\boldsymbol{\xi}_{i-1}+\mathbf{C}_{0,0}^{-1}\mathbf{A}_{1,-1}(h_{i}w^{-}+w^{+}c_{i})\boldsymbol{\theta}],
\end{array}
\end{displaymath}}
as $i\to\infty$, where $\boldsymbol{\xi}_{0}:=\boldsymbol{\xi}$ using results from Lemma \ref{lemp}.
\end{proof}
Proposition \ref{prop1} and Lemmas \ref{lemp}, \ref{lemp1} provide the necessary information to prove the convergence of series given in Theorem \ref{theorem}:\vspace{-0.1in}
\begin{proposition}\label{pr11}\begin{enumerate}
\item For $m\geq0$, $n\geq1$, $\sum_{i=0}^{\infty}h_{i}\gamma_{i}^{m}\delta_{i}^{n}$, $\sum_{i=0}^{\infty}c_{i+1}\gamma_{i+1}^{m}\delta_{i}^{n}$ converge absolutely.
\item For $m\geq0$, $\sum_{i=0}^{\infty}\gamma_{i}^{m}\xi_{i,k}$ where $\boldsymbol{\xi}_{0}:=\boldsymbol{\xi}=(\xi_{0,0},\xi_{0,1})^{T}$, $\boldsymbol{\xi}_{i}:=(\xi_{i,0},\xi_{i,1})^{T}$, $i\geq 1$ converge absolutely.
\end{enumerate}
\end{proposition}
\begin{proof}
Since $\boldsymbol{\theta}$ is independent of the values of $\gamma_{i}$, $\delta_{i}$, let $\theta_{0}=\mu$. %Then, \eqref{form} implies that
%\small{\begin{equation}
%\begin{array}{rl}
%q_{m,n}(0)\propto&\mu\sum_{i=0}^{\infty}(h_{i}\gamma_{i}^{m}+c_{i+1}\gamma_{i+1}^{m})\delta_{i}^{n},\,m\geq0,n\geq1,\\
%q_{m,n}(1)\propto	&(\lambda+2\alpha)\sum_{i=0}^{\infty}(h_{i}\gamma_{i}^{m}+c_{i+1}\gamma_{i+1}^{m})\delta_{i}^{n},\,m\geq0,n\geq1,\\
%q_{m,0}(k)\propto	&\sum_{i=0}^{\infty}\gamma_{i}\xi_{i,k},\,m\geq1,\,k=0,1,\vspace{-0.1in}
%\end{array}
%\label{fgh}
%\end{equation}}
The proof follows the lines in \cite{ad0}. Set for $m\geq0$, $n\geq1$,
\small{\begin{equation*}
\begin{array}{rl}
R_{1}(m,n):=&\lim_{i\to\infty}\left|\frac{h_{i+1}\gamma_{i+1}^{m}\delta_{i+1}^{n}}{h_{i}\gamma_{i}^{m}\delta_{i}^{n}}\right|=\lim_{i\to\infty}\left|\frac{\frac{h_{i+1}}{c_{i+1}}\frac{\gamma_{i+1}^{m}}{\delta_{i+1}}\frac{\delta_{i+1}^{m+n}}{\gamma_{i}^{m+n}}}{\frac{h_{i}}{c_{i+1}}\frac{\gamma_{i}^{m}}{\delta_{i}^{m}}\frac{\delta_{i}^{m+n}}{\gamma_{i}^{m+n}}}\right|,\\R_{2}(m,n):=&\lim_{i\to\infty}\left|\frac{c_{i+2}\gamma_{i+2}^{m}\delta_{i+1}^{n}}{c_{i+1}\gamma_{i+1}^{m}\delta_{i}^{n}}\right|=\lim_{i\to\infty}\left|\frac{\frac{c_{i+2}}{h_{i+1}}\frac{\gamma_{i+2}^{m}}{\delta_{i+1}^{m}}\frac{\delta_{i+1}^{m+n}}{\gamma_{i+1}^{m+n}}}{\frac{c_{i+1}}{h_{i+1}}\frac{\gamma_{i+1}^{m}}{\delta_{i}^{m}}\frac{\delta_{i}^{m+n}}{\gamma_{i+1}^{m+n}}}\right|,
\end{array}
\label{colll}
\end{equation*}}
since $h_{i}$s, $c_{i}$s and $\gamma_{i}$s, $\delta_{i}$s are non-zero. If these limits exist and are less than one, then the series in assertion $1.$ converge absolutely. 

Using Lemmas \ref{lemp}, \ref{lemp1}, $R_{1}(m,n)=R_{2}(m,n)=\frac{\frac{\lambda+\mu}{\mu}-s^{-}}{s^{+}-\frac{\lambda+\mu}{\mu}}\left(\frac{w^{-}}{w^{+}}\right)^{m+n-1}$,
where $s^{\pm}=1+\rho\pm\sqrt{1+\rho^{2}}$ due to \eqref{mmm}. Note that $w^{-}<w^{+}$. Moreover, simple algebraic arguments yields that $\frac{\frac{\lambda+\mu}{\mu}-s^{-}}{s^{+}-\frac{\lambda+\mu}{\mu}}<1$ if and only is $\lambda>0$, which is true.% First note that $s^{+}-\frac{\lambda+\mu}{\mu}>0$. %Indeed, 
%\small{\begin{displaymath}\begin{array}{c}
%s^{+}-\frac{\lambda+\mu}{\mu}>0\Leftrightarrow \rho+\sqrt{1+\rho^{2}}>\frac{\lambda}{\mu}\Leftrightarrow \frac{\lambda}{2\alpha}+\sqrt{(\frac{\mu}{\lambda})^{2}+(\frac{\lambda+2\alpha}{2\alpha})^{2}}>0,\vspace{-0.1in}
%\end{array}
%\end{displaymath}}
%which is true. 
%Thus, $\frac{\frac{\lambda+\mu}{\mu}-s^{-}}{s^{+}-\frac{\lambda+\mu}{\mu}}<1\Leftrightarrow \frac{\lambda+\mu}{\mu}-s^{-}<s^{+}-\frac{\lambda+\mu}{\mu}\Leftrightarrow 2\frac{\lambda+\mu}{\mu}<s^{-}+s^{+}=2(1+\rho)
%\Leftrightarrow 1<\frac{\lambda+2\alpha}{2\alpha}\Leftrightarrow \lambda>0$ which is also true.  
  Therefore $R_{1}(m,n)=R_{2}(m,n)<1$, and the series in assertion $1.$ converge absolutely, i.e., the error terms that we introduce in each vertical and horizontal compensation step tend to zero. Using similar arguments we can prove assertion $2.$
\end{proof}

We conclude by showing that the series
\small{\begin{displaymath}
\begin{array}{rl}
\mathcal{C}:=&\sum_{k=0}^{1}\sum_{m=0}^{\infty}\sum_{n=1}^{\infty}q_{m,n}(k)+\sum_{k=0}^{1}\sum_{m=0}^{\infty}q_{m,0}(k)\\
=&\sum_{k=0}^{1}\sum_{m=0}^{\infty}\sum_{n=1}^{\infty}q_{m,n}(k)+\frac{1}{1+\rho}\sum_{m=1}^{\infty}(\rho q_{m-1,1}(1)+\frac{1}{2}q_{m,1}(0)+\alpha\rho q_{m,1}(1))+\frac{\alpha}{\lambda}(\frac{\mu}{\lambda}q_{0,1}(0)+q_{0,1}(1)).\vspace{-0.1in}
\end{array}
\end{displaymath}}
converges. Thus, the convergence follows if $\sum_{k=0}^{1}\sum_{m=0}^{\infty}\sum_{n=1}^{\infty}q_{m,n}(k)<\infty$.
Note that\vspace{-0.1in}
\small{\begin{displaymath}
\begin{array}{rl}
\sum_{k=0}^{1}\sum_{m=0}^{\infty}\sum_{n=1}^{\infty}|q_{m,n}(k)|\leq&(\lambda+2\alpha+\mu)\sum_{m=0}^{\infty}\sum_{n=1}^{\infty}[\sum_{i=0}^{\infty}(|h_{i}\gamma_{i}^{m}\delta_{i}^{n}|+|c_{i+1}\gamma_{i+1}^{m}\delta_{i}^{n}|)]\\<&[\sum_{i=0}^{\infty}\frac{|h_{i}|}{1-|\gamma_{i}|}\frac{|\delta_{i}|}{1-|\delta_{i}|}+\sum_{i=0}^{\infty}\frac{|c_{i+1}|}{1-|\gamma_{i+1}|}\frac{|\delta_{i}|}{1-|\delta_{i}|}].
\end{array}
\end{displaymath}}
The convergence of these series follows directly as in Proposition \ref{pr11}, having in mind that as $i\to\infty$, $\gamma_{i}\to 0$, $\delta_{i}\to 0$, and using lemmas \ref{lemp}, \ref{lemp1}.
%
%To show that these series converge, we only need to show that the following limits are less than one:
%\small{\begin{displaymath}
%\begin{array}{l}
%R_{3}:=\lim_{i\to\infty}\left|\frac{\frac{|h_{i+1}|}{1-|\gamma_{i+1}|}\frac{|\delta_{i+1}|}{1-|\delta_{i+1}|}}{\frac{|h_{i}|}{1-|\gamma_{i}|}\frac{|\delta_{i}|}{1-|\delta_{i}|}}\right|=\lim_{i\to\infty}\left|\frac{\frac{|h_{i+1}|}{|c_{i+1}|}\frac{1}{1-|\gamma_{i+1}|}\frac{|\delta_{i+1}|}{|\gamma_{i+1}|}\frac{1}{1-|\delta_{i+1}|}}{\frac{|h_{i}|}{|c_{i+1}|}\frac{1}{1-|\gamma_{i}|}\frac{|\delta_{i}|}{|\gamma_{i+1}|}\frac{1}{1-|\delta_{i}|}}\right|,\,
%R_{4}:=\lim_{i\to\infty}\left|\frac{\frac{|c_{i+2}|}{1-|\gamma_{i+2}|}\frac{|\delta_{i+1}|}{1-|\delta_{i+1}|}}{\frac{|c_{i+1}|}{1-|\gamma_{i+1}|}\frac{|\delta_{i}|}{1-|\delta_{i}|}}\right|=\lim_{i\to\infty}\left|\frac{\frac{|c_{i+2}|}{|h_{i+1}|}\frac{1}{1-|\gamma_{i+2}|}\frac{|\delta_{i+1}|}{|\gamma_{i+1}|}\frac{1}{1-|\delta_{i+1}|}}{\frac{|c_{i+1}|}{|h_{i+1}|}\frac{1}{1-|\gamma_{i+1}|}\frac{|\delta_{i}|}{\gamma_{i+1}}\frac{1}{1-|\delta_{i}|}}\right|.\vspace{-0.1in}
%\end{array}
%\end{displaymath}}
%Having in mind that as $i\to\infty$, $\gamma_{i}\to 0$, $\delta_{i}\to 0$, and using lemmas \ref{lemp}, \ref{lemp1} we have that, 
%$R_{3}=R_{4}=\frac{\frac{\lambda+\mu}{\mu}-s^{-}}{s^{+}-\frac{\lambda+\mu}{\mu}}<1$.
Therefore, $\sum_{k=0}^{1}\sum_{m=0}^{\infty}\sum_{n=1}^{\infty}q_{m,n}(k)<\infty$. To conclude, the series in \eqref{form} is the unique up to a multiplicative constant, solution of the balance equations \eqref{e1}-\eqref{e6}. The following theorem states the main result of this paper:
\begin{theorem}\label{theoremf}
For $\rho<1$,
\small{\begin{equation*}
\begin{array}{rl}
\mathbf{q}(m,n)=&\mathcal{C}^{-1}\sum_{i=0}^{\infty}(h_{i}\gamma_{i}^{m}+c_{i+1}\gamma_{i+1}^{m})\delta_{i}^{n}\boldsymbol{\theta},\,\,({\text{pairs with the same $\delta$-term})},\,m\geq0,n\geq1,\\
=&\mathcal{C}^{-1}(h_{0}\gamma_{0}^{m}\delta_{0}^{n}+\sum_{i=0}^{\infty}(h_{i+1}\delta_{i}^{n}+c_{i+1}\delta_{i+1}^{n})\gamma_{i+1}^{m}\boldsymbol{\theta},\,\,({\text{pairs with the same $\gamma$-term})}\,m\geq0,n\geq1,\\
\mathbf{q}(m,0)=&\mathcal{C}^{-1}(\gamma_{0}^{m}\boldsymbol{\xi}+\sum_{i=1}^{\infty}\gamma_{i}^{m}\boldsymbol{\xi}_{i}),\,m\geq1,\\
\mathbf{q}(0,0)=&-\mathbf{A}_{0,0}^{-1}\mathbf{A}_{0,-1}\mathbf{q}(0,1),\vspace{-0.1in}
\end{array}
\label{formf}
\end{equation*}}
and $\boldsymbol{\theta}=\theta_{0}(1,\frac{\lambda+2\alpha}{\mu})^{T}$, $\theta_{0}>0$, $\mathcal{C}$ be the normalization constant, and $\{\gamma_{i}\}_{i\in\mathbb{N}}$, $\{\delta_{i}\}_{i\in\mathbb{N}}$, $\{h_{i}\}_{i\in\mathbb{N}}$, $\{c_{i}\}_{i\in\mathbb{N}}$, $\{\boldsymbol{\xi}_{i}\}_{i\in\mathbb{N}}$, are
obtained recursively based on the analysis above.\vspace{-0.2in}
\end{theorem}
\section{Numerical example}\label{sec:num}\vspace{-0.1in}
In this section we provide a simple numerical illustration. In Table \ref{tab} we list the probabilities $q_{0,n}:=q_{0,n}(0)+q_{0,n}(1)$, $n=0,1,2,3$ that are computed with a precision of $10^{-10}$. The number in parentheses denotes the number of required terms to attain that accuracy. We observe that the number of required terms decrease rapidly. In Figure \ref{ssst} we plot the probability of an empty system, i.e., $q_{0,0}(0)$ for increasing values of $\lambda$. As expected, $q_{0,0}(0)$ decreases as $\lambda$ increases. However, we can also observe that by increasing the retrial rate $\alpha$ from 5 to 8 and 10, $q_{0,0}(0)$ takes larger values, but still retained the decreasing trend. This is expected since the more we increase $\alpha$, the more we increase the chances of retrial jobs to be connected with the server, and thus, orbit queues empty faster.\vspace{-0.15in}  
\begin{figure}[h]
\centering
\includegraphics[scale=0.45]{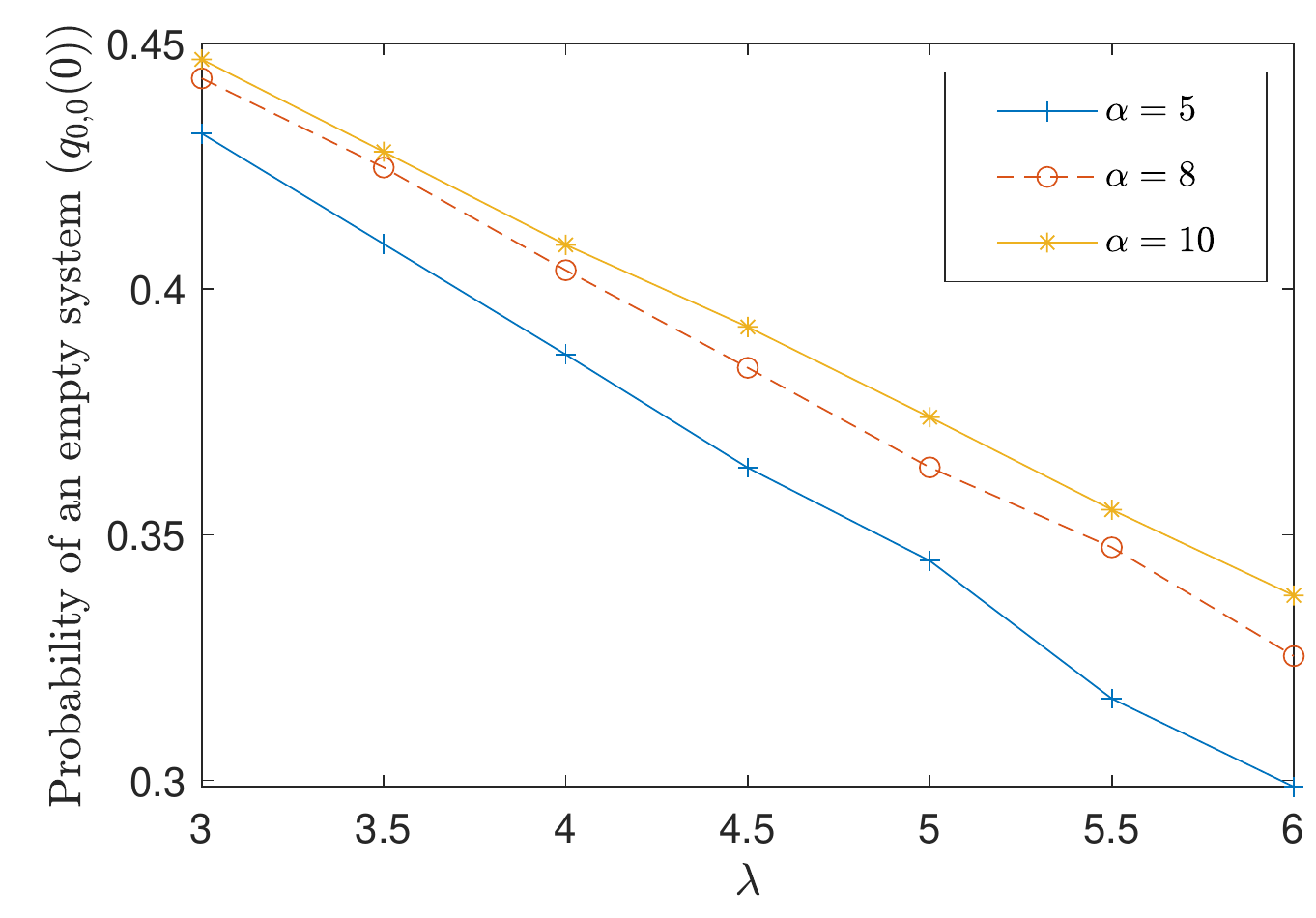}\vspace{-0.15in}
\caption{The probability of an empty system ($q_{0,0}(0)$) for $\mu=10$.}
\label{ssst}
\end{figure}\vspace{-0.3in}
\small{\begin{table}[h]
\caption{Computation of several probabilities for increasing values of $\lambda$ and $\mu=10$, $\alpha=3$.} % title of Table
\centering % used for centering table
\begin{tabular}{|c|c| c| c| c| c|} % centered columns (4 columns)
\hline\hline %inserts double horizontal lines
$\lambda$& $q_{0,0}$& $q_{0,1}$&$q_{0,2}$&$q_{0,3}$\\ [0.5ex] % inserts table
%heading
\hline % inserts single horizontal line
$2$ & $0.5639\ (159)$ & $0.0063\ (159)$&$0.1235\ (8)$&$0.2496\ (4)$\\ \hline% inserting body of the table
$3$ & $0.5437\ (79)$& $0.0125\ (79)$&$0.0986\ (10)$&$0.1992\ (5)$ \\\hline
$4$&$0.5056\ (51)$&$0.0193\ (51)$&$0.0895\ (11)$&$0.1658\ (6)$ \\% [1ex] adds vertical space
\hline %inserts single line
\end{tabular}
\label{tab} % is used to refer this table in the text
\end{table}}\vspace{-0.15in}
%\section{Conclusion}\label{sec:conc}
%In this work, we introduced the JSQ policy in the retrial setting. The model at hand is described by a random walk in the quarter plane modulated by a two-state Markovian process. We investigate the \textit{ergodicity conditions} using truncation arguments and study its \textit{stationary tail decay rate}. Then, we applied the \textit{compensation method} to study its stationary behaviour. Our work serves as a building block to apply the compensation method in even general Markov modulated two-dimensional queueing models. Moreover, it serves as a first step in the stationary analysis of even general retrial models operating under the JSQ policy. In a future work, we plan to apply this methodology to obtain the equilibrium distribution in the case more complicated arrival/service processes, as well as in case of priority queueing. These features will reveal additional technical requirements and further mathematical challenges. 
\appendix\vspace{-0.1in}
\section*{Acknowledgements}\vspace{-0.1in}
We would like to thank the Editors and the reviewers for the insightful remarks, which helped to improve the original exposition.\vspace{-0.1in}
\section{Proof of Theorem \ref{dec}}\label{appc}\vspace{-0.1in}
Since $\rho<1$, $\{X(t)\}$ has a unique stationary distribution $\mathbf{q}=\{\mathbf{q}^{T}(m),m\geq0\}$, where $\mathbf{q}^{T}(m)=\{\mathbf{q}^{T}(m,n)=(q_{m,n}(0),q_{m,n}(1)),n\geq0\}$, and using $T_{-1}$, $T_{0}$, $T_{1}$ we generate an 1-arithmetic Markov-additive process. Then to prove Theorem \ref{dec} we have to show (see e.g., \cite[Proposition 3.1]{sakuma} that there exist positive vectors $\mathbf{y}$, $\mathbf{p}$, such that \vspace{-0.1in}
\begin{equation}
\begin{array}{cccc}
\mathbf{p}(\rho^{-2}T_{1}+T_{0}+\rho^{2}T_{-1})=\mathbf{0},&\mathbf{p}\mathbf{y}<\infty,&
(\rho^{-2}T_{1}+T_{0}+\rho^{2}T_{-1})\mathbf{y}=\mathbf{0},&\mathbf{q}^{T}(0)\mathbf{y}<\infty.
\end{array}\label{test}
\end{equation}
Define the $S_{*}\times S_{*}$ matrix $K=\rho^{2}T_{-1}+T_{0}+\rho T_{1}$, i.e.,
\begin{displaymath}
\begin{array}{c}
K=\begin{pmatrix}
K_{0}&\bar{K}_{1}&&&\\
K_{-1}&K_{0}&K_{1}&&\\
&K_{-1}&K_{0}&K_{1}&\\
&&\ddots&\ddots&\ddots
\end{pmatrix},\,K_{0}=\mathbf{C}_{0,0}^{T},\,\bar{K}_{1}=2\rho^{2}\mathbf{A}_{-1,1}^{T}+\mathbf{A}_{0,1}^{T},\,K_{1}=\rho^{2}\mathbf{A}_{-1,1}^{T},\,K_{-1}=\rho^{-2}\mathbf{A}_{1,-1}^{T}+\mathbf{A}_{0,-1}^{T}
\end{array}
\end{displaymath}

The following lemma shows that we can find a positive vector $\mathbf{y}=\{\rho^{-n}\mathbf{v},\,n\geq0\}$, such that $K\mathbf{y}=\mathbf{0}$.
\begin{lemma}
Let $\mathbf{v}=(1,\frac{\mu(\lambda+2\alpha)}{\lambda(\lambda+\mu+2\alpha)})^{T}$, such that $\mathbf{y}=\{\rho^{-n}\mathbf{v},\,n\geq0\}$. Then, $\mathbf{y}$ is positive such that $K\mathbf{y}=\mathbf{0}$.
\end{lemma}
\begin{proof}\vspace{-0.1in}
Note that $K\mathbf{y}=\mathbf{0}$ implies,\vspace{-0.1in}
\small{\begin{displaymath}
\begin{array}{rl}
K_{0}\mathbf{v}+\bar{K}_{1}\rho^{-1}\mathbf{v}=&\rho^{-1}[\mathbf{C}_{0,0}^{T}\rho+2\rho^{2}\mathbf{A}_{-1,1}^{T}+\mathbf{A}_{0,1}^{T}]\mathbf{v}\\
=&\rho^{-1}\begin{pmatrix}
-(\lambda+2\alpha)\rho&\rho(\lambda+2\alpha\rho)\\
\mu\rho&\lambda-(\lambda+\mu)\rho
\end{pmatrix}\begin{pmatrix}
1\\
\frac{\mu(\lambda+2\alpha)}{\lambda(\lambda+\mu+2\alpha)}
\end{pmatrix}=\mathbf{0},\,n=0,\\
\rho^{-n}[\rho K_{-1}+K_{0}+\rho K_{1}]\mathbf{v}=&\rho^{-n-1}[\mathbf{A}_{1,-1}^{T}+\rho^{2}(\mathbf{A}_{-1,1}^{T}+\mathbf{A}_{0,-1}^{T})+\rho\mathbf{C}_{0,0}]\mathbf{v}\\=&\rho^{-n-1}[\mathbf{A}_{0,1}^{T}+2\rho^{2}\mathbf{A}_{-1,1}^{T}+\rho\mathbf{C}_{0,0}]\mathbf{v}=\mathbf{0},\,n\geq1.\vspace{-0.1in}
\end{array}
\end{displaymath}}
Therefore, $K\mathbf{y}=\mathbf{0}$, and $\mathbf{y}$ is positive since $\mathbf{v}$ is positive.\vspace{-0.1in}
\end{proof}

We now construct a positive vector $\mathbf{p}=\{\mathbf{p}_{n},n\geq0\}$ such that $\mathbf{p}K=\mathbf{0}$. Let $\Delta_{v}$ be the diagonal matrix whose diagonal elements are the corresponding elements of $\mathbf{v}$. Let the diagonal matrix $D=diag(\Delta_{v},\rho^{-1}\Delta_{v},\rho^{-2}\Delta_{v},\ldots)$, and denote $K_{D}=D^{-1}KD$. %Then,
%\small{\begin{displaymath}
%\begin{array}{c}
%K_{D}=\begin{pmatrix}
%\Delta_{v}^{-1}K_{0}\Delta_{v}&\rho^{-1}%\Delta_{v}^{-1}\bar{K}_{1}\Delta_{v}&&&\\\rho\Delta_{v}^{-1}K_{-1}\Delta_{v}&\Delta_{v}^{-1}K_{0}\Delta_{v}&\rho^{-1}\Delta_{v}^{-1}K_{1}\Delta_{v}&&&\\
%&\rho\Delta_{v}^{-1}K_{-1}\Delta_{v}&\Delta_{v}^{-1}K_{0}\Delta_{v}&\rho^{-1}\Delta_{v}^{-1}K_{1}\Delta_{v}&&\\
%&&\rho\Delta_{v}^{-1}K_{-1}\Delta_{v}&\Delta_{v}^{-1}K_{0}\Delta_{v}&\rho^{-1}\Delta_{v}^{-1}K_{1}\Delta_{v}&\\
%&&\ddots&\ddots&\ddots
%\end{pmatrix}.\vspace{-0.1in}
%\end{array}
%\end{displaymath}}
Note that $K_{D}\mathbf{1}=D^{-1}KD\mathbf{1}=D^{-1}K\mathbf{y}=\mathbf{0}$. Thus, $K_{D}$ is a transition rate matrix of a QBD with finite phases at each level. We now check the ergodicity of $K_{D}$. Let $\mathbf{u}$ the stationary probability vector of
$\Delta_{v}^{-1}[\rho K_{-1}+K_{0}+\rho^{-1}K].
$
The mean drift at internal states is
\small{\begin{displaymath}
\begin{array}{rl}
\mathbf{u}[\rho^{-1}\Delta_{v}^{-1}K_{1}\Delta_{v}-\rho\Delta_{v}^{-1}K_{-1}\Delta_{v}]=&\rho\mathbf{u}\Delta_{v}^{-1}[\rho^{-2}K_{1}-K_{-1}]\mathbf{v}
=-\rho\mathbf{u}\Delta_{v}^{-1}\begin{pmatrix}
0&0\\
0&\lambda\rho^{-2}
\end{pmatrix}\mathbf{v}<0.\vspace{-0.1in}
\end{array}
\end{displaymath}}
Since $K_{D}$ is ergodic \cite{neuts}, there exists a stationary distribution $\bar{\xi}=\{\bar{\xi}_{n},n\geq0\}$ such that $\bar{\xi}K_{D}=\mathbf{0}$, or equivalently, since $D$ is invertible, $\bar{\xi}D^{-1}K=\mathbf{0}$. The next lemma summarizes the construction of $\mathbf{p}$.
\begin{lemma}\label{lem8}
Let $\mathbf{p}:=\bar{\xi}D^{-1}=\{\rho^{n}\bar{\xi}_{n}\Delta_{v}^{-1},n\geq0\}$. Then $\mathbf{p}$ is a positive vector satisfying $\mathbf{p}K=\mathbf{0}$ and $\mathbf{p}\mathbf{y}<\infty$.
\end{lemma}
\begin{proof}
We only need to show that $\mathbf{p}\mathbf{y}<\infty$. Indeed, $\mathbf{p}\mathbf{y}=\rho^{n}\bar{\xi}_{n}\Delta_{v}^{-1}\rho^{-n}\Delta_{v}\mathbf{1}=\bar{\xi}_{n}\mathbf{1}=1<\infty$.
\end{proof}
It remains to verify that $\mathbf{q}^{T}(0)\mathbf{y}<\infty$. This is a direct consequence of the following general result. 
\begin{lemma}
For any small $\epsilon>0$, $\limsup_{m\to\infty}\rho^{(-2+\epsilon)m}\mathbf{q}(m,0)=\mathbf{0}$.
\end{lemma}
\begin{proof}
The proof is based on employing truncation arguments as in \cite[Lemma 5.3]{sakuma}. To proceed, we have to modify the original system described by $\{X(t)\}$, such that the shortest orbit queue will not retry when the difference of the two orbit queues attains a predefined constant $M\geq3$. Namely, the state $(m,n,k)\in S$ is truncated as $n\leq M$ by removing the state transitions from $(m,M,k)$ to $(m-1,M+1,k)$. For this modified model, the longer orbit queue is set for the level instead of the shortest queue. Let $Q_{i}^{M}(t)$ the queue length in orbit $i$ for the truncated model, and  $X^{M}(t)=\{(max(Q_{1}^{M}(t),Q_{2}^{M}(t)),|Q_{1}^{M}(t)-Q_{2}^{M}(t)|,C(t));t\geq0\}$ is a QBD with complex boundary behaviour. For this truncated longer orbit queue model we can prove a series of results that show that the decay rate of the original model is $\rho^2$. % and repeating blocks $T_{-1}^{(M)}$, $T_{1}^{(M)}$ and $T_{0}^{(M)}$, which are $2(M+1)\times 2(M+1)$ matrices given by
%\begin{displaymath}
%\begin{array}{c}
%\small{\mathbf{Q}^{(M)}=\begin{pmatrix}
%T_{0,0}^{(M)}&T_{0,1}^{(M)}&&&\\
%T_{1,0}^{(M)}&T_{1,1}^{(M)}&T_{1,2}^{(M)}&&\\
%&\ddots&\ddots&\ddots&\\
%&&T_{M-1,M-2}^{(M)}&T_{M-1,M-1}^{(M)}&T_{M-1,M}^{(M)}&\\
%&&&T_{M,M-1}^{(M)}&T_{M,M}^{(M)}&T_{1}^{(M)}&\\
%&&&&T_{-1}^{(M)}&T_{0}^{(M)}&T_{1}^{(M)}\\
%&&&&&\ddots&\ddots&\ddots\\
%\end{pmatrix},}
%\end{array}
%\end{displaymath}
%where $T_{i,i}^{(M)}$, $T_{i,i+1}^{(M)}$ and $T_{i+1,i}^{(M)}$, $i=0,1,\ldots,M-1$ are $2(i+1)\times 2(i+1)$, $2(i+1)\times 2(i+2)$ and $2(i+2)\times 2(i+1)$ matrices respectively. Moreover, $T_{M,M}^{(M)}$ is $2(M+1)\times 2(M+1)$, and $T_{-1}^{(M)}$, $T_{1}^{(M)}$ and $T_{0}^{(M)}$ are also $2(M+1)\times 2(M+1)$ where its repeating blocks are
%\small{\begin{displaymath}
%\begin{array}{l}
%T_{-1}^{(M)}=\begin{pmatrix}
%\mathbb{O}&&&\\
%\mathbf{A}^{T}_{0,-1}&\mathbb{O}&&\\
%&\mathbf{A}^{T}_{0,-1}&\mathbb{O}&\\
%&&\ddots&\ddots&\\
%&&&\mathbf{A}^{T}_{0,-1}&\mathbb{O}\\
%\end{pmatrix},\,T_{1}^{(M)}=\begin{pmatrix}
%\mathbb{O}&\mathbf{A}^{T}_{1,-1}&&\\
%&\mathbb{O}&\mathbb{O}&&\\
%&&\ddots&\ddots&\\
%&&&\mathbb{O}&\mathbb{O}
%\end{pmatrix},\,T_{0}^{(M)}=\begin{pmatrix}
%\mathbf{C}^{T}_{0,0}&2\mathbf{A}^{T}_{-1,1}&&&\\
%\mathbf{A}^{T}_{1,-1}&\mathbf{C}^{T}_{0,0}&\mathbf{A}^{T}_{-1,1}&&\\
%&\ddots&\ddots&\ddots&\\
%&&\mathbf{A}^{T}_{1,-1}&\mathbf{C}^{T}_{0,0}&\mathbf{A}^{T}_{-1,1}\\
%&&&\mathbf{A}^{T}_{1,-1}&\mathbf{B}^{T}_{0,0}
%\end{pmatrix}.
%\end{array}
%\end{displaymath}}
Due to space constraints the rest of the proof is omitted. For more details see \cite[pp. 196-199]{sakuma}.\end{proof}\vspace{-0.15in}
\bibliographystyle{spmpsci}
\bibliography{arxiv_versionv2}


%Author biography without author photo.
%Author biography. Author biography. Author biography.
%Author biography. Author biography. Author biography.

%\bio{figs/pic1}
%Author biography with author photo.
%Author biography. Author biography. Author biography.
\end{document}